\documentclass[a4paper,12pt]{article}
\usepackage{latexsym}
\usepackage{amsmath}
\usepackage{amssymb}
\usepackage{enumerate}
\usepackage{bm}
\usepackage{mathrsfs}
\usepackage{color}

\usepackage{theorem}
\newtheorem{theorem}{Theorem}[section]
\newtheorem{proposition}[theorem]{Proposition}
\newtheorem{lemma}[theorem]{Lemma}
\newtheorem{corollary}[theorem]{Corollary}
\newtheorem{case}{Case}
\newtheorem{mtheorem}{Main Theorem}

\theorembodyfont{\rmfamily}
\newtheorem{proof}{\textmd{\textit{Proof.}}}

\newtheorem{remark}[theorem]{Remark}

\newtheorem{definition}[theorem]{Definition}
\newtheorem{step}{Step}

\makeatletter

\@addtoreset{equation}{section}
\makeatother

\newcommand{\qedd}{\hfill \Box}
\newcommand{\ve}{\varepsilon}
\newcommand{\del}{\partial}
\newcommand{\lra}{\longrightarrow}
\newcommand{\e}{\mathrm{e}}
\newcommand{\g}{\mathrm{g}}

\newcommand{\N}{\ensuremath{\mathbb{N}}}

\newcommand{\R}{\ensuremath{\mathbb{R}}}
\newcommand{\Sph}{\ensuremath{\mathbb{S}}}
\newcommand{\cC}{\ensuremath{\mathcal{C}}}
\newcommand{\cI}{\ensuremath{\mathcal{I}}}

\newcommand{\cP}{\ensuremath{\mathcal{P}}}

\newcommand{\fm}{\ensuremath{\mathfrak{m}}}

\newcommand{\sP}{\ensuremath{\mathsf{P}}}

\def\vol{\mathop{\mathrm{vol}}\nolimits}
\def\diam{\mathop{\mathrm{diam}}\nolimits}

\def\supp{\mathop{\mathrm{supp}}\nolimits}

\def\Ent{\mathop{\mathrm{Ent}}\nolimits}
\def\Hess{\mathop{\mathrm{Hess}}\nolimits}
\def\Var{\mathop{\mathrm{Var}}\nolimits}
\def\Ric{\mathop{\mathrm{Ric}}\nolimits}

\def\CD{\mathop{\mathrm{CD}}\nolimits}
\def\RCD{\mathop{\mathrm{RCD}}\nolimits}

\title{Quantitative estimates for the Bakry--Ledoux isoperimetric inequality}
\author{Cong Hung MAI\thanks{
Department of Mathematics, Kyoto University, Kyoto 606-8502, Japan
({\sf hongmai@math.kyoto-u.ac.jp})}
\and
Shin-ichi OHTA\thanks{
Department of Mathematics, Osaka University, Osaka 560-0043, Japan
({\sf s.ohta@math.sci.osaka-u.ac.jp})} \textsuperscript{,}\thanks{
RIKEN Center for Advanced Intelligence Project (AIP),
1-4-1 Nihonbashi, Tokyo 103-0027, Japan}}

\date{\today}
\begin{document}

\maketitle

\begin{abstract}
We establish a quantitative isoperimetric inequality
for weighted Riemannian manifolds with $\Ric_{\infty} \ge 1$.
Precisely, we give an upper bound of the volume of the symmetric difference between
a Borel set and a sub-level (or super-level) set of the associated guiding function
(arising from the needle decomposition),
in terms of the deficit in Bakry--Ledoux's Gaussian isoperimetric inequality.
This is the first quantitative isoperimetric inequality on noncompact spaces
besides Euclidean and Gaussian spaces.
Our argument makes use of Klartag's needle decomposition (also called localization),
and is inspired by a recent work of Cavalletti, Maggi and Mondino on compact spaces.
Besides the quantitative isoperimetry,
a reverse Poincar\'e inequality for the guiding function that we have as a key step,
as well as the way we use it, are of independent interest.
\end{abstract}


\section{Introduction}

Geometric and functional inequalities under various curvature bounds
are one of the main subjects of comparison geometry and geometric analysis.
Beyond an inequality itself, its \emph{rigidity}
(characterizing a space attaining equality, that we call a model space) and
\emph{stability} (showing that the space is close to the model space when equality nearly holds)
are important subjects, for instance in connection with the theory of convergence of spaces.
One of the classical stability results is Colding's `almost sphere theorem' \cite{Co1,Co2};
see Remark~\ref{rm:quan} for some more (classical and recent) results.
There are at least two strategies for stability problems.
One is based on compactness arguments:
we take a sequence of spaces asymptotically satisfying equality in the inequality in question,
and apply a rigidity result to its limit space.
This method usually provides implicit estimates.
Another strategy is an explicit quantitative estimate that we follow in this article.

Quantitative isoperimetric inequalities were intensively studied
in the Euclidean spaces (\cite{FiMP,FuMP}) and Gaussian spaces (\cite{BBJ,CFMP,El,MN}).
For the Gaussian space $(\R^n,\bm{\gamma}^n)$, $\bm{\gamma}^n :=(2\pi)^{-n/2}\e^{-|x|^2/2}\,dx$,
it is known that isoperimetric minimizers are half-spaces.
Precisely, given $\theta \in (0,1)$,
a half-space $H_{w,a_{\theta}}:=\{ x \in \R^n \,|\, \langle x,w \rangle \le a_{\theta} \}$
with $w \in \Sph^{n-1}$ and $a_{\theta} \in \R$ satisfying $\bm{\gamma}^n(H_{w,a_{\theta}})=\theta$
attains the minimum perimeter $\cI_{(\R^n,\bm{\gamma}^n)}(\theta)$
among sets with volume $\theta$.
Note that the isoperimetric profile $\cI_{(\R^n,\bm{\gamma}^n)}$ is independent of $n$,
and we will denote $\bm{\gamma}^1$ by $\bm{\gamma}$.
In \cite{BBJ,El} it was shown that
\begin{equation}\label{eq:Gauss}
\min_{w \in \Sph^{n-1}} \bm{\gamma}^n(A \,\triangle\, H_{w,a_{\theta}})
 \le C(\theta) \sqrt{\sP(A)-\cI_{(\R,\bm{\gamma})}(\theta)}
\end{equation}
holds for $A \subset \R^n$ with $\bm{\gamma}^n(A)=\theta$,
where $A \,\triangle\, B$ is the symmetric difference of $A$ and $B$
and $\sP(A)$ denotes the perimeter of $A$ with respect to $\bm{\gamma}^n$.
We call $\delta:=\sP(A)-\cI_{(\R,\bm{\gamma})}(\theta)$
the \emph{deficit} in the isoperimetric inequality.
Note that the order $\sqrt{\delta}$ in \eqref{eq:Gauss} is independent of $n$,
and is known to be optimal.

In curved spaces (such as Riemannian manifolds) without any symmetry nor homogeneity,
much less is known for quantitative isoperimetric inequalities.
For instance, for the L\'evy--Gromov isoperimetric inequality whose model space is a sphere,
the rigidity was classical whereas there had not been any quantitative version until recently.
A breakthrough was made by Klartag \cite{Kl},
who established an alternative proof of isoperimetric inequalities not relying on
the deep regularity theory from geometric measure theory.
The method developed in \cite{Kl} is the \emph{needle decomposition}
(also called the \emph{localization}, see Subsection~\ref{ssc:needle}),
which has its roots in convex geometry and enables us to reduce an inequality
on a (high-dimensional) space into those on geodesics (called \emph{needles}).
Then one only needs to perform a $1$-dimensional analysis on geodesics,
which is much simpler especially for isoperimetric inequalities.
This technique turned out useful also in rigidity and stability problems.

In \cite{CM1}, Cavalletti and Mondino generalized the needle decomposition
to essentially non-branching metric measure spaces
satisfying the \emph{curvature-dimension condition}
$\CD(K,N)$ with $K \in \R$ and $N \in (1,\infty)$,
and established the L\'evy--Gromov isoperimetric inequality,
as well as its rigidity for $\RCD(K,N)$-spaces.
The curvature-dimension condition $\CD(K,N)$ is a synthetic notion
of the lower Ricci curvature bound,
equivalent to $\Ric_N \ge K$ for weighted Riemannian or Finsler manifolds,
and the \emph{Riemannian curvature-dimension condition} $\RCD(K,N)$
is its reinforced version coupled with the linearity of heat flow
(see Subsection~\ref{ssc:wRic}).
Then, with a deeper analysis via the needle decomposition,
Cavalletti, Maggi and Mondino \cite{CMM} investigated
the stability for $\CD(N-1,N)$-spaces $(X,d,\fm)$ with $N \in (1,\infty)$.
They showed that, for $A \subset X$ with $\fm(A)=\theta$,
\begin{equation}\label{eq:CMM}
\fm\big( A \,\triangle\, B_r(x) \big) \le C(N,\theta) \big( \sP(A)-\cI_N(\theta) \big)^{N/(N^2+2N-1)}
\end{equation}
holds for some $x \in X$, where $B_r(x)$ is the ball of center $x$ and radius $r$,
with the model isoperimetric profile $\cI_N$ and appropriate $r=r(N,\theta)$.
This means that $A$ is close to a ball in terms of $\fm$.
We refer to \cite{CES} for another quantitative study of isoperimetric inequalities
on closed Riemannian manifolds with a different method.

The aim of this article is to explore the possibility of applying the needle decomposition
to a quantitative isoperimetric inequality under $\Ric_{\infty} \ge K>0$.
In general, stability problems in terms of $\Ric_{\infty}$ are more challenging
than those of $\Ric_N$ with $N \in (1,\infty)$,
since Gromov's precompactness theorem (\cite[\S 5.A]{Gr}) does not apply.
Without loss of generality we assume $K=1$ in the sequel.
In this case, Bakry and Ledoux \cite{BL} showed an isoperimetric inequality
with the Gaussian space $(\R,\bm{\gamma})$ as the model space
(see \cite{AM,Oisop} for some generalizations).
One of the most important differences between $\Ric_N \ge N-1$ (or $\CD(N-1,N)$)
and $\Ric_{\infty} \ge 1$ from our viewpoint is that $\Ric_N \ge N-1$ implies the compactness
(precisely, the diameter is bounded above by $\pi$ by the Bonnet--Myers theorem),
while $\Ric_{\infty} \ge 1$ can hold for noncompact spaces.
In fact, the model space for $\Ric_N \ge N-1$ is a sphere
and some stability estimates in terms of the diameter were essentially used in \cite{CMM}.
In the case of $\Ric_{\infty} \ge 1$, the possible unboundedness of needles causes several difficulties.
We perform careful estimates on needles to overcome these difficulties
(see for example Section~\ref{sc:psi}), and our main theorem asserts the following.

\begin{mtheorem}[Theorem~\ref{th:main}]
Let $(M,g,\fm)$ be a complete weighted Riemannian manifold
such that $\Ric_{\infty} \ge 1$ and $\fm(M)=1$.
Fix $\theta \in (0,1) \setminus \{1/2\}$ and $\ve \in (0,1)$,
take a Borel set $A \subset M$ with $\fm(A)=\theta$,
and assume that $\sP(A) \le \cI_{(\R,\bm{\gamma})}(\theta)+\delta$ holds
for sufficiently small $\delta>0$ $($relative to $\theta$ and $\ve)$.
Then, for the guiding function $u$ associated with $A$ such that $\int_M u \,d\fm=0$, we have
\[ {\min}\Big\{ \fm\big( A \,\triangle\, \{ u \le a_{\theta} \} \big),
 \fm\big( A \,\triangle\, \{ u \ge a_{1-\theta} \} \big) \Big\}
 \le C(\theta,\ve) \delta^{(1-\ve)/(9-3\ve)}. \]
\end{mtheorem}

Here the \emph{guiding function} $u$ stems from the construction of the needle decomposition
(see Subsection~\ref{ssc:needle}).
In the rigidity case, an isoperimetric minimizer is in fact
given as a sub-level set of the associated guiding function (see Theorem~\ref{th:Morgan}).
Furthermore, the guiding function is somehow related to the Busemann function,
hence its sub-level set can be viewed as `a half-space' or `a ball with center at infinity'
(see Remark~\ref{rm:main}\eqref{main-u} for a further account).
Therefore our main theorem is regarded as a counterpart to
\eqref{eq:Gauss} as well as \eqref{eq:CMM}.
We refer to Remark~\ref{rm:main} for further discussions and related open problems.
Here we only remark that the case of $\theta=1/2$ is removed merely for technical reasons
(Remark~\ref{rm:main}\eqref{main-1/2}),
and the main theorem holds true also for reversible Finsler manifolds
(Remark~\ref{rm:main}\eqref{main-F}).

Our careful calculation on needles provides further applications.
We in particular show that the guiding function $u$ in the theorem
enjoys the \emph{reverse Poincar\'e inequality}
\[ \Var_{(M,\fm)}(u) \ge \frac{1}{\Lambda'(\theta,\ve,\delta)} \int_M |\nabla u|^2 \,d\fm \]
such that $\Lambda'(\theta,\ve,\delta) \le (1-C(\theta,\ve)\delta^{(1-\ve)/(3-\ve)})^{-1}$
(Theorem~\ref{th:revP}).
This is indeed a reverse form of the Poincar\'e inequality
\[ \Var_{(M,\fm)}(u) \le \int_M |\nabla u|^2 \,d\fm \]
induced from $\Ric_{\infty} \ge 1$ (see \eqref{eq:Poin}).
The use of the reverse Poincar\'e inequality is inspired by \cite{Ma2} where we studied the rigidity problem,
and reveals an interesting relation between the isoperimetric inequality and the spectral gap
via the guiding function.
The reverse Poincar\'e inequality plays an essential role
to integrate $1$-dimensional estimates on needles into an estimate on $M$
in the proof of the main theorem (precisely, Proposition~\ref{pr:var_X} is a key ingredient).
We refer to Remark~\ref{rm:quan}(b) for related stability and rigidity results
for functional inequalities, and to \cite{DMS,Ha} for more recent results.

The article is organized as follows.
In Section~\ref{sc:prel} we review necessary notions related to the weighted Ricci curvature,
isoperimetric inequalities, and the needle decomposition.
Then Sections~\ref{sc:psi}--\ref{sc:revP} are devoted to the $1$-dimensional analysis.
We first establish in Section~\ref{sc:psi} that a small deficit in the isoperimetric inequality
implies that the measure on the needle is close to the Gaussian one (Proposition~\ref{pr:psi}).
This is the starting point of all the estimates in the sequel.
In Section~\ref{sc:symm} we show that a small deficit in the isoperimetric inequality
implies a small symmetric difference from a half-space (Proposition~\ref{pr:symm}).
In Section~\ref{sc:revP} we establish a reverse Poincar\'e inequality on needles
(Proposition~\ref{pr:revP}).
Coming back to Riemannian manifolds,
in Section~\ref{sc:rev} we derive a reverse Poincar\'e inequality for a guiding function
(Theorem~\ref{th:revP}) from the reverse Poincar\'e inequality on needles in the previous section.
Finally, we prove Main Theorem (Theorem~\ref{th:main}) in Section~\ref{sc:main}.

\section{Preliminaries}\label{sc:prel}

Throughout the article, let $(M,g)$ be a connected, complete $\cC^{\infty}$-Riemannian manifold
of dimension $n \ge 2$ without boundary.
We denote by $d$ the Riemannian distance function.
A \emph{weighted Riemannian manifold} means a triple $(M,g,\fm)$,
where $\fm = \e^{-\Psi}\vol_g$ is a measure modifying the Riemannian volume measure
$\vol_g$ of $(M,g)$ with a \emph{weight function} $\Psi \in \cC^{\infty}(M)$.

\subsection{Weighted Ricci curvature and spectral gap}\label{ssc:wRic}

On $(M,g,\fm)$, we need to modify the Ricci curvature $\Ric_g$ with respect to $g$
taking into account the behavior of $\fm$ (namely $\Psi$).

\begin{definition}[Weighted Ricci curvature]\label{df:wRic}
Given $v \in T_xM$ and $N \in \R \setminus \{n\}$,
define the \emph{weighted Ricci curvature} $\Ric_N(v)$
(also called the \emph{Bakry--\'{E}mery--Ricci curvature}) by
\[ \Ric_{N}(v) :=\Ric_g(v) +\Hess \Psi(v,v) -\frac{\langle \nabla \Psi(x),v\rangle^2}{N-n}. \]
As the limits of $N \to \infty$ and $N \downarrow n$, we also define
\begin{align*}
\Ric_{\infty}(v) &:=\Ric_g(v) +\Hess \Psi(v,v), \\
\Ric_{n}(v) &:=\begin{cases} \Ric_g(v) +\Hess \Psi(v,v)
 & \text{if}\ \langle \nabla\Psi(x),v\rangle = 0,\\
 -\infty & \text{otherwise}.
 \end{cases}
\end{align*}
\end{definition}

Note that $\Ric_{N}(cv) :=c^{2}\Ric_{N}(v)$ for all $c \in \R$.
If $\Psi$ is constant, then $\Ric_N(v)$ coincides with $\Ric_g(v)$ for all $N$.
We will write $\Ric_{N}\ge K$ for $K \in \R$ when $\Ric_{N}(v) \ge K|v|^2$ holds for all $v \in TM$.
Several remarks on $\Ric_N$ are in order.

\begin{remark}\label{rm:wRic}
\begin{enumerate}[(a)]
\item
By definition $\Ric_N$ enjoys the monotonicity
\[ \Ric_n(v) \le \Ric_N(v) \le \Ric_{\infty}(v) \le \Ric_{N'}(v) \]
for $N \in [n,\infty)$ and $N' \in (-\infty,n)$.
Therefore, for example, $\Ric_{\infty} \ge K$ is a weaker condition than
$\Ric_N \ge K$ with $N \in [n,\infty)$.

\item
The case of $N \in [n,\infty]$ has been intensively investigated by Bakry and his collaborators
in the context of \emph{$\Gamma$-calculus} (see \cite{BGL}),
including the isoperimetric inequality under $\Ric_{\infty} \ge K>0$ by Bakry--Ledoux \cite{BL}.
The study of the case of $N \in (-\infty,n)$ is rather recent,
we refer to \cite{GZ,KM,Ma1,Ma2,Mi2,Oneg,Wy} among others.

\item
The lower curvature bound $\Ric_{N} \ge K$ is known to be equivalent to
the \emph{curvature-dimension condition} $\CD(K,N)$ in the sense of Lott--Sturm--Villani,
see \cite{CMS,LV,vRS, StI,StII,Vi} (as well as \cite{Oint} for a Finsler analogue).
Metric measure spaces satisfying $\CD(K,N)$ (\emph{$\CD(K,N)$-spaces} for short)
share many analytic and geometric properties
with weighted Riemannian or Finsler manifolds of $\Ric_N \ge K$.
Moreover, requiring an additional condition on the linearity of heat flow,
one can introduce a reinforced version called the \emph{Riemannian curvature-dimension condition}
$\RCD(K,N)$ (see \cite{AGS,EKS}).
This excludes Finsler manifolds and we can show, for instance,
a Cheeger--Gromoll-type splitting theorem \cite{Gi-split,Gi-ov}.
\end{enumerate}
\end{remark}

We will also make use of the Laplacian associated with $\fm$.

\begin{definition}[Weighted Laplacian]\label{df:Lap}
The \emph{weighted Laplacian} (also called the \emph{Witten Laplacian})
acting on $u \in \cC^\infty(M)$ is defined by
\[ \Delta_{\fm} u := \Delta u - \langle \nabla u,\nabla\Psi \rangle, \]
where $\Delta$ is the canonical Laplacian with respect to $g$.
\end{definition}

The integration by parts formula for $\vol_g$ readily implies that for $\fm$, namely
\[ \int _{M} \phi \Delta_{\fm}u \,d\fm =-\int _{M}\langle \nabla \phi,\nabla u \rangle \,d\fm \]
holds for $\phi \in \cC^{\infty}(M)$ with compact support.

If $\Ric_{\infty} \ge K > 0$, then $\fm$ has a Gaussian decay and $\fm(M)<\infty$ holds
(\cite[Theorem~4.26]{StI}).
Since adding a constant to $\Psi$ does not change $\Ric_{\infty}$,
we can normalize $\fm$ as $\fm(M)=1$ without loss of generality.
From $\Ric_{\infty} \ge K>0$ we also have a lower bound of the first nonzero eigenvalue $\lambda_1$
of $-\Delta_{\fm}$ as $\lambda_{1} \ge K$.
This is a generalization of the classical Lichnerowicz inequality to the $\Ric_{\infty}$ context,
and equivalent to the \emph{Poincar\'e inequality}
\begin{equation}\label{eq:Poin}
\int_M u^2 \,d\fm -\bigg( \int_M u \,d\fm \bigg)^2
 \le \frac{1}{K} \int_M |\nabla u|^2 \,d\fm.
\end{equation}
The LHS of \eqref{eq:Poin} is the \emph{variance} of $u$
and will be denoted by $\Var_{(M,\fm)}(u)$.
The equality case was studied in \cite[Theorem 2]{CZ} as follows,
as a counterpart to the classical Obata theorem in \cite{Ob}.
 
\begin{theorem}[Rigidity of spectral gap]\label{th:CZ}
Let $(M,g,\fm)$ be a complete weighted Riemannian manifold satisfying $\fm(M)=1$
and $\Ric_{\infty} \ge K$ for some $K>0$.
If equality $\lambda_{1} = K$ is achieved with an eigenfunction $u$,
then we have the following.
\begin{enumerate}[{\rm (i)}]
\item
$(M,g,\fm)$ is isometric to the product space $\R \times \Sigma$ as weighted Riemannian manifolds,
where $\Sigma=u^{-1}(0)$ and $(\Sigma,g_{\Sigma},\fm_{\Sigma})$ is
an $(n-1)$-dimensional weighted Riemannian manifold of $\Ric_{\infty} \ge K$,
and $\R$ is equipped with the Gaussian measure $\sqrt{K/(2\pi)}\e^{-Kx^2/2} \,dx$.

\item
The function $u$ is constant on $\{t\} \times \Sigma$ for each $t \in \R$,
and we can moreover choose as $u(t,x)=t$.
\end{enumerate}
\end{theorem}

We remark that $u$ being an eigenfunction with eigenvalue $K$
implies equality in \eqref{eq:Poin} with $\int_M u \,d\fm=0$.
We refer to \cite{GKKO} for a generalization of Theorem~\ref{th:CZ} to $\RCD(K,\infty)$-spaces,
and to \cite{Ma1} for the case of $\Ric_N \ge K>0$ with $N<-1$
where we have a warped product splitting of hyperbolic nature instead of the isometric splitting.

\subsection{Isoperimetric inequalities}\label{ssc:isop}

An important geometric result on weighted Riemannian manifolds with lower Ricci curvature bounds
is an isoperimetric inequality.
In order to state the isoperimetric inequality,
we define the \emph{perimeter} of a Borel set $A \subset M$ with $\fm(A)<\infty$ as
\begin{equation}\label{eq:peri}
\sP(A) :=\inf_{\{\phi_i\}_{i \in \N}}  \liminf_{i \to \infty} \int_M |\nabla \phi_i| \,d\fm,
\end{equation}
where $\{\phi_i\}_{i \in \N}$ runs over all sequences of Lipschitz functions
converging to the characteristic function $\chi_A$ of $A$ in $L^1(\fm)$.
When $\sP(A)<\infty$, we have $\sP(M \setminus A)=\sP(A)$.

One can also consider the \emph{Minkowski exterior content} (or \emph{boundary measure})
defined by
\[ \fm^{+}(A) := \liminf_{\ve \to 0} \frac{\fm(B(A,\ve) \setminus A)}{\ve} \]
for a Borel set $A$, where $B(A,\ve)$ denotes the open $\ve$-neighborhood of $A$.
By taking $\phi_i(x):=\max\{ 1-i \cdot d(A,x),0 \}$,
we see that $\fm^+(A) \ge \sP(A)$ in general.
If the boundary $\del A$ is sufficiently smooth, then $\fm^+(A) =\sP(A)$ holds
and they coincide with $(\e^{-\Psi} \mathcal{H}^{n-1})(\del A)$,
where $\mathcal{H}^{n-1}$ is the $(n-1)$-dimensional Hausdorff measure.
This is the case for isoperimetric minimizers by virtue of the regularity theory
in geometric measure theory (see \cite[\S 2.2]{Mi1} for instance).
We refer to \cite[Section~3.3]{AFP}, \cite[Section~14]{BZ} and \cite{ADG} for more on the perimeter.

Assuming $\fm(M) = 1$, we define the \emph{isoperimetric profile} as
 \[ \cI_{(M,\fm)}(\theta) := \inf \{ \sP(A) \,|\, A \subset M,\, \fm(A)=\theta \} \]
for $\theta \in (0,1)$, where $A$ runs over all Borel sets with $\fm(A)=\theta$.
An isoperimetric inequality under the condition $\Ric_{\infty} \ge K> 0$
was first shown by Bakry--Ledoux \cite{BL}, having the same form as that for the Gaussian spaces.
Milman \cite{Mi1,Mi2} then intensively studied the combination of $\Ric_N \ge K$
and a diameter bound $\diam(M) \le D$, and showed the following.

\begin{theorem}[Isoperimetric inequalities]\label{th:isop}
Let $(M,g,\fm)$ be a complete weighted Riemannian manifold satisfying $\fm(M)=1$,
$\diam(M) \le D$ with $D \in (0,\infty]$, and $\Ric_{N} \ge K$ for $N \in (-\infty,0) \cup [n,\infty]$
and $K \in \R$.
Then we have $\cI_{(M,\fm)}(\theta) \ge \cI_{(K,N,D)}(\theta)$ for all $\theta \in (0,1)$,
where $\cI_{(K,N,D)}$ is an explicitly given function depending only on $K,N$ and $D$.
\end{theorem}

The estimation is sharp in all the parameters $K,N$ and $D$,
and we stress that $\cI_{(K,N,D)}$ is independent of the dimension $n$ of $M$.
We refer to \cite{Mi1,Mi2} for precise formulas of the function $\cI_{(K,N,D)}$.
Here we present the only relevant case in this article, namely $K>0$ and $N=\infty$, for later use.
Without the diameter bound $(D=\infty)$, the model space is the Gaussian space and
\[ \cI_{(K,\infty,\infty) }(\theta) =\sqrt{\frac{K}{2\pi}} \e^{-Ka_{\theta}^2/2}, \qquad
 \text{where}\ \ \sqrt{\frac{K}{2\pi}} \int_{-\infty}^{a_{\theta}} \e^{-Kt^2/2} \,dt =\theta. \]
For $D \in (0,\infty)$, we have
\[ \cI_{(K,\infty,D)}(\theta) =\inf_{\xi \in [-D,0]} f_{\xi,D}(\theta) \]
with 
\[ f_{\xi,D}(\theta) :=\frac{\e^{-Kb_{\theta,\xi,D}^2/2}}{\int_{\xi}^{\xi+D} \e^{-Kt^2/2} \,dt}, \qquad
 \text{where}\ \ \frac{\int_{\xi}^{b_{\theta,\xi,D}} \e^{-Kt^2/2} \,dt}{\int_{\xi}^{\xi+D} \e^{-Kt^2/2} \,dt}
 =\theta. \]
Let us have a closer look on how the diameter influences the isoperimetric profile,
with the help of some calculations in \cite[Lemma~3.1]{Ma2}.

\begin{lemma}[Difference between $\cI_{(K,\infty,D)}$ and $\cI_{(K,\infty,\infty)}$]\label{lm:I_D}
Let $K,D \in (0,\infty)$.
For $\theta \in (0,1)$, we have
\[ \cI_{(K,\infty,D)}(\theta) -\cI_{(K,\infty,\infty)}(\theta)
 >\frac{\sqrt{K}}{\pi} \frac{\e^{-KD^2}}{\sqrt{K}D+1}. \]
\end{lemma}

\begin{proof}
Let us abbreviate in this proof as $\cI_D=\cI_{(K,\infty,D)}$ and $\cI_{\infty}=\cI_{(K,\infty,\infty)}$.
We deduce from the definitions of $a_{\theta}$ and $b_{\theta,\xi,D}$ that
\begin{equation}\label{eq:a'}
\cI_{\infty}(\theta) \frac{da_{\theta}}{d\theta}
 =f_{\xi,D}(\theta) \frac{\del b_{\theta,\xi,D}}{\del\theta} =1.
\end{equation}
Then we find
\begin{align}
\cI'_{\infty}(\theta) &=-K a_{\theta} \cI_{\infty}(\theta) \frac{da_{\theta}}{d\theta}
 =-Ka_{\theta}, \label{eq:I'}\\
f'_{\xi,D}(\theta) &= -K b_{\theta,\xi,D} f_{\xi,D}(\theta) \frac{\del b_{\theta,\xi,D}}{\del\theta}
 =-Kb_{\theta,\xi,D}. \nonumber
\end{align}
Now, in order to estimate $f_{\xi,D}(\theta) -\cI_{\infty}(\theta)$, it suffices to test
at $\theta=0,1$ and $\theta_0$ at where $\cI'_{\infty}(\theta_0)=f'_{\xi,D}(\theta_0)$ holds.
At such $\theta_0$, it follows from the calculations above that $a_{\theta_0}=b_{\theta_0,\xi,D}$,
thereby
\[ f_{\xi,D}(\theta_0) -\cI_{\infty}(\theta_0)
 =\left( \frac{1}{\int_{\xi}^{\xi+D} \e^{-Kt^2/2} \,dt} -\sqrt{\frac{K}{2\pi}} \right)
 \e^{-Kb_{\theta_0,\xi,D}^2/2}. \]
Noticing $|b_{\theta_0,\xi,D}|<D$ as well as $\cI_{\infty}(0)=\cI_{\infty}(1)=0$, we obtain
\[ \cI_D(\theta) -\cI_{\infty}(\theta)
 >\left( \frac{1}{\int_{-D}^D \e^{-Kt^2/2} \,dt} -\sqrt{\frac{K}{2\pi}} \right) \e^{-KD^2/2} \]
for all $\theta \in (0,1)$.
Since
\[ \int_{-D}^D \e^{-Kt^2/2} \,dt
 =\frac{1}{\sqrt{K}} \int_{-\sqrt{K}D}^{\sqrt{K}D} \e^{-s^2/2} \,ds
 =\frac{1}{\sqrt{K}} \bigg( \sqrt{2\pi}-2\int_{\sqrt{K}D}^{\infty} \e^{-s^2/2} \,ds \bigg) \]
and
\[ \int_{\sqrt{K}D}^{\infty} \e^{-s^2/2} \,ds
 \ge \int_{\sqrt{K}D}^{\infty} \frac{s^2+s+1}{(s+1)^2} \e^{-s^2/2} \,ds
 =-\bigg[ \frac{\e^{-s^2/2}}{s+1} \bigg]_{\sqrt{K}D}^{\infty}
 =\frac{\e^{-KD^2/2}}{\sqrt{K}D+1}, \]
we find
\[ \frac{1}{\int_{-D}^D \e^{-Kt^2/2} \,dt} -\sqrt{\frac{K}{2\pi}}
 \ge \sqrt{\frac{K}{2\pi}} \bigg\{ 
 \bigg( 1-\sqrt{\frac{2}{\pi}} \frac{\e^{-KD^2/2}}{\sqrt{K}D+1} \bigg)^{-1} -1 \bigg\}
 \ge \frac{\sqrt{K}}{\pi} \frac{\e^{-KD^2/2}}{\sqrt{K}D+1}. \]
This completes the proof.
$\qedd$
\end{proof}

Note that the lower bound in Lemma~\ref{lm:I_D} is uniform in $\theta$.
From the calculation in the above proof, we also find a fundamental fact that
the profile $\cI_{(K,\infty,\infty)}$ is strictly concave.

\begin{lemma}[Concavity of $\cI_{(K,\infty,\infty)}$]\label{lm:conc}
For $\theta \in (0,1)$, we have
\[ \cI''_{(K,\infty,\infty)}(\theta) =-\frac{K}{\cI_{(K,\infty,\infty)}(\theta)}.  \]
\end{lemma}

\begin{proof}
This is staightforward from \eqref{eq:I'} and \eqref{eq:a'}.
$\qedd$
\end{proof}

Let us close the subsection with a rigidity result of Morgan \cite[Theorem~18.7]{Mo}
(see \cite[Section~3]{Ma2} for an alternative proof based on the needle decomposition).

\begin{theorem}[Rigidity of isoperimetric inequality]\label{th:Morgan}
Let $(M,g,\fm)$ be a complete weighted Riemannian manifold satisfying $\fm(M)=1$
and $\Ric_{\infty} \ge K$ for some $K>0$.
If $\fm^+(A) =\cI_{(K,\infty,\infty)}(\theta)$ holds for some $A \subset M$ with $\theta=\fm(A) \in (0,1)$,
then we have the following.
\begin{enumerate}[{\rm (i)}]
\item
$(M,g,\fm)$ is isometric to the product space $\R \times \Sigma$ as weighted Riemannian manifolds,
where $(\Sigma,g_{\Sigma},\fm_{\Sigma})$ is an $(n-1)$-dimensional weighted Riemannian manifold
of $\Ric_{\infty} \ge K$, and $\R$ is equipped with the Gaussian measure
$\sqrt{K/(2\pi)}\e^{-Kx^2/2} \,dx$.

\item
The set $A$ is a half-space in this product structure, in the sense that
$A$ coincides with $(-\infty,a_{\theta}] \times \Sigma$ or $[a_{1-\theta},\infty) \times \Sigma$.
\end{enumerate}
\end{theorem}

Our main theorem (Theorem~\ref{th:main}) will be a quantitative version of this theorem.
Notice that the first assertion on the splitting phenomenon is same as Theorem~\ref{th:CZ}.
In fact, in \cite{Ma2}, we saw that the guiding function $u$ associated with the set $A$
(see the next subsection) turns out providing the sharp spectral gap $\lambda_1=K$,
and the isoperimetric minimizer $A$ is in fact a sub-level (or super-level) set of $u$.
These facts motivate reverse Poincar\'e inequalities below
(Proposition~\ref{pr:revP}, Theorem~\ref{th:revP})
as well as the formulation of Theorem~\ref{th:main}.

\subsection{Needle decompositions}\label{ssc:needle}

Now we recall the main ingredient of our argument,
the \emph{needle decomposition} (also called the \emph{localization}),
established on weighted Riemannian manifolds by the seminal work of Klartag \cite{Kl}.
The needle decomposition has its roots in convex geometry,
going back to \cite{PW} and developed in \cite{GM,KLS,LS}.
Roughly speaking, via the needle decomposition
one can reduce an inequality on a high-dimensional space to those on geodesics
(\emph{needles}) in that space.
Then, especially in isoperimetric inequalities, the $1$-dimensional analysis on geodesics
could be simpler than the direct analysis on the original space.

We first define transport rays associated with a $1$-Lipschitz function.
We say that a function $u:M \lra \R$ is \emph{$1$-Lipschitz}
if $|u(x)-u(y)| \le d(x,y)$ holds for all $x,y \in M$.

\begin{definition}[Transport rays]\label{df:T-ray}
Let $u$ be a 1-Lipschitz function on $M$.
We say that $X \subset M$ is a \emph{transport ray} associated with $u$
if $|u(x)-u(y)| = d(x,y)$ holds for all $x,y \in X$ and if, for all $z \not\in X$,
there exists $x \in X$ such that $|u(x) - u(z)|<d(x,z)$.
\end{definition}

Any transport ray is a closed set and necessarily the image of a minimal geodesic,
thereby equipped with the natural distance structure and identified with a closed interval.
We shall make use of the following kind of needle decomposition (\cite[Theorems~1.2, 1.5]{Kl}),
where $u$ is called the \emph{guiding function} acting as a `guide' of the decomposition.

\begin{theorem}[Needle decomposition]\label{th:Ndl}
Let $(M,g,\fm)$ be a complete weighted Riemannian manifold satisfying $\Ric_N \ge K$,
and take a function $f \in L^1(\fm)$ such that $\int_{M} f \,d\fm = 0$ and
$\int_{M}|f(x)|d(x_{0},x) \,\fm(dx) < \infty$ for some $x_0 \in M$.
Then there exists a $1$-Lipschitz function $u$ on $M$, a partition $\{X_q\}_{q \in Q}$ of $M$,
a measure $\nu$ on $Q$ and a family of probability measures $\{\fm_q\}_{q \in Q}$ on $M$
satisfying the following.
\begin{enumerate}[{\rm (i)}]
\item\label{Ndl1}
For any measurable set $A \subset M$, we have $\fm(A) = \int_{Q} \fm_q(A) \,\nu(dq)$.
Moreover, for $\nu$-almost every $q \in Q$, we have $\supp(\fm_q)=X_q$.

\item\label{Ndl2}
For $\nu$-almost every $q \in Q$, $X_q$ is a transport ray associated with $u$.
Moreover, if $X_q$ is not a singleton, then the weighted Ricci curvature of
$(X_q,|\cdot|,\fm_q)$ satisfies $\Ric_N \ge K$.

\item\label{Ndl3}
For $\nu$-almost every $q \in Q$, we have $\int_{X_q} f \,d\fm_q =0$.
\end{enumerate}
\end{theorem}

The first assertion \eqref{Ndl1} includes the measurability of $\fm_q(A)$ in $q \in Q$.
We also observe from \eqref{Ndl1} that $\nu(Q)=\fm(M)$.
In \eqref{Ndl2}, by denoting $\fm_q=\e^{-\psi} \,dx$ along $X_q$,
$\psi$ is smooth on the interior of $X_q$ and $\Ric_N \ge K$ means that $\psi''  \ge K+(\psi')^2/(N-1)$.

Our argument on quantitative isoperimetric inequalities is indebted to
Klartag's proof in \cite{Kl} of the isoperimetric inequality (Theorem~\ref{th:isop})
by the needle decomposition.
Let us recall it for later convenience.

Let $(M,g,\fm)$ satisfy $\Ric_N \ge K$, $\diam(M) \le D$ and $\fm(M) = 1$.
Given $\theta \in (0,1)$, we fix an arbitrary Borel set $A \subset M$ with $\fm(A) = \theta$.
Consider the function $f(x) := \chi_{A}(x) -\theta$.
Then we find $\int_{M} f \,d\fm = 0$,
and obtain $(Q,\nu)$ and $\{(X_q,\fm_q)\}_{q \in Q}$
associated with $f$ as in Theorem~\ref{th:Ndl}.
Note that \eqref{Ndl3} in Theorem~\ref{th:Ndl} implies $\fm_q(A) =\theta$ for $\nu$-almost every $q \in Q$
(and hence $X_q$ is not a singleton).
Moreover, $(X_q,|\cdot|,\fm_q)$ enjoys $\CD(K,N)$ by \eqref{Ndl2} and clearly $\diam(X_q) \le D$.
Therefore, the $1$-dimensional isoperimetric inequality yields
$\sP(A \cap X_q) \ge  \cI_{(K,N,D)}(\theta)$ for $\nu$-almost every $q \in Q$,
where $\sP(A \cap X_q)$ denotes the perimeter of $A \cap X_q$ in $(X_q,|\cdot|,\fm_q)$.
Together with Lemma~\ref{lm:Lem4.1} below, we conclude that
\[ \sP(A) \ge \int_Q \sP(A \cap X_q) \,\nu(dq) \ge \cI_{(K,N,D)}(\theta). \]
Taking the infimum in $A$ completes the proof of $\cI_{(M,\fm)}(\theta) \ge \cI_{(K,N,D)}(\theta)$.

\begin{remark}[Regularity of $\psi$]\label{rm:Ndl}
As we mentioned above, thanks to \cite{Kl},
$\fm_q$ has a smooth density and $\Ric_N \ge K$ is regarded as $\psi'' \ge K+(\psi')^2/(N-1)$.
For our purpose, however, the weak formulation $\CD(K,N)$ is sufficient.
In the case of $N=\infty$, $\CD(K,\infty)$ is equivalent to $\psi'' \ge K$ in the weak sense
(also called the \emph{$K$-convexity}), namely
\begin{equation}\label{eq:Kcon}
\psi\big( (1-t)x+ty \big) \le (1-t)\psi(x) +t\psi(y) -\frac{K}{2}(1-t)td(x,y)^2
\end{equation}
for all $x,y \in X_q$ and $t \in (0,1)$.
In the non-smooth framework of essentially non-branching $\CD(K,N)$-spaces as in \cite{CM1,CMM},
one cannot expect the smoothness and only the weak formulation makes sense.
\end{remark}

\section{Difference of weight functions}\label{sc:psi}

Henceforth, we normalize as $K=1$ without loss of generality.
In this and the following two sections, we work on $1$-dimensional spaces enjoying $\CD(1,\infty)$,
appearing as needles in Theorem~\ref{th:Ndl}.
Let $I \subset \R$ be a (bounded or unbounded) closed interval
equipped with a measure $\fm=\e^{-\psi}\, dx$,
where $dx$ denotes the $1$-dimensional Lebesgue measure
and $\psi$ is a locally Lipschitz function.
Then, as we mentioned in Remark~\ref{rm:Ndl},
for $(I,|\cdot|,\fm)$ satisfying $\CD(1,\infty)$ means that $\psi$ is $1$-convex as in \eqref{eq:Kcon},
\[ \psi\big( (1-t)x+ty \big) \le (1-t)\psi(x) +t\psi(y) -\frac{(1-t)t}{2}|x-y|^2 \]
for any $x,y \in I$ and $t \in (0,1)$.
The following useful property due to Bobkov (\cite[Proposition~2.1]{Bob}) is then available.

\begin{lemma}\label{lm:Bobkov}
Let $\fm=\e^{-\psi}\, dx$ be a probability measure on a closed interval $I \subset \R$
such that $\psi$ is convex.
Then the minimum of $\sP(A)$ on the class of all Borel sets $A \subset I$ with $\fm(A) = \theta$
coincides with the minimum on the subclass consisting of
$($semi-infinite$)$ intervals $(-\infty,a] \cap I$ and $[b,\infty) \cap I$.
\end{lemma}

We remark that what follows from \cite{Bob} is the analogous assertion for $\fm^+$,
however, one can see that its minimum coincides with that of $\sP$ by,
for instance, \cite[Theorem~3.6]{ADG}.
We shall compare $\fm$ on $I$ with the Gaussian measure $\bm{\gamma}$ on $\R$
with mean $0$ and variance $1$, denoted by
\[ \bm{\gamma} :=\frac{1}{\sqrt{2\pi}} \e^{-x^2/2} \,dx =\e^{-\bm{\psi}_{\g}(x)} \,dx,
 \qquad \bm{\psi}_{\g}(x) :=\log\big( \sqrt{2\pi} \big) +\frac{1}{2}x^2. \]
Recall from Subsection~\ref{ssc:isop} that
the isoperimetric profile of $(\R,|\cdot|,\bm{\gamma})$ is given by
\[ \cI_{(\R,\bm{\gamma})}(\theta) =\cI_{(1,\infty,\infty)}(\theta) =\e^{-\bm{\psi}_{\g}(a_{\theta})}, \qquad
 \theta=\bm{\gamma}\big( (-\infty,a_{\theta}] \big).\]
Our goal in this section is to show the following core estimate.

\begin{proposition}[Difference of weight functions]\label{pr:psi}
Let $I \subset \R$ be a closed interval equipped with a probability measure
$\fm=\e^{-\psi} \,dx$ such that $\psi$ is $1$-convex.
Fix $\theta \in (0,1)$ and assume that
\begin{equation}\label{eq:a_theta}
\int_{I \cap (-\infty,a_{\theta}]} \e^{-\psi} \,dx =\theta
\end{equation}
and that
\begin{equation}\label{eq:deficit}
\e^{-\psi(a_{\theta})} \le \e^{-\bm{\psi}_{\g}(a_{\theta})} +\delta
\end{equation}
holds for sufficiently small $\delta>0$ $($relative to $\theta)$.
Then we have
\begin{equation}\label{eq:delta>}
\psi(x)-\bm{\psi}_{\g}(x)
 \ge \big( \psi'_+(a_{\theta})-a_{\theta} \big) (x-a_{\theta}) -\omega(\theta) \delta
\end{equation}
for every $x \in I$, and
\begin{equation}\label{eq:delta<}
\psi(x)-\bm{\psi}_{\g}(x)
 \le \big( \psi'_+(a_{\theta})-a_{\theta} \big) (x-a_{\theta}) +\omega(\theta) \sqrt{\delta}
\end{equation}
for every $x \in [S,T]$ such that $\lim_{\delta \to 0}S=-\infty$ and $\lim_{\delta \to 0}T=\infty$,
where $\psi'_+$ denotes the right derivative of $\psi$ and
$\omega(\theta)$ is a constant depending only on $\theta$.
\end{proposition}

Note that \eqref{eq:a_theta} is achieved by translating $I$ in $\R$.
Moreover, thanks to Lemma~\ref{lm:Bobkov}, we can assume
$\cI_{(I,\fm)}(\theta) =\e^{-\psi(a_{\theta})}$ by reversing $I$ if necessary.
Since $\cI_{(I,\fm)}(\theta) \ge \cI_{(\R,\bm{\gamma})}(\theta)=\e^{-\bm{\psi}_{\g}(a_{\theta})}$
holds in general (Theorem~\ref{th:isop}),
$\delta>0$ represents the \emph{deficit} in this isoperimetric inequality.
See \eqref{eq:Tdef} and \eqref{eq:Sdef} below for the precise choices of $T$ and $S$,
as well as \eqref{eq:gT}, \eqref{eq:eT} and \eqref{eq:geS} for their asymptotic behaviors as $\delta \to 0$.
Finally, we stress that the lower bound \eqref{eq:delta>} holds on whole $I$
whereas the upper bound \eqref{eq:delta<} is valid only on $[S,T]$.
This is natural since the decay of $\fm$ near infinity does not effect the isoperimetric profile
and thus can be arbitrarily fast.

\begin{proof}
We will denote by $\psi'_+$ (resp.\ $\psi'_-$) the right (resp.\ left) derivative of $\psi$.
The $1$-convexity of $\psi$ implies that $\psi'_+$ and $\psi'_-$ always exist
and $\psi'_- \le \psi'_+$ holds.
Put $I_-:=I \cap (-\infty,a_{\theta}]$ and $I_+:=I \cap [a_{\theta},\infty)$.

\begin{step}[$\psi'_+(a_{\theta})$ and $a_{\theta}$]
We first estimate the difference of $\psi'_+(a_{\theta})$ and
$\bm{\psi}'_{\g}(a_{\theta})=a_{\theta}$.
We deduce from the $1$-convexity of $\psi$ and the hypothesis \eqref{eq:deficit} that, for $x \in I_+$,
\begin{align}
\psi(x)
&\ge \psi(a_{\theta}) +\psi'_+(a_{\theta})(x-a_{\theta}) +\frac{(x-a_{\theta})^2}{2} \nonumber\\
&\ge -\log(\e^{-\bm{\psi}_{\g}(a_{\theta})}+\delta)
 +\psi'_+(a_{\theta})(x-a_{\theta}) +\frac{(x-a_{\theta})^2}{2}.
\label{eq:psi(x)}
\end{align}
Hence we have
\begin{align}
1-\theta &=\int_{I_+} \e^{-\psi} \,dx \nonumber\\
&\le (\e^{-\bm{\psi}_{\g}(a_{\theta})}+\delta)
 \int_{I_+} \exp\bigg( {-}\psi'_+(a_{\theta})(x-a_{\theta})
 -\frac{(x-a_{\theta})^2}{2} \bigg) \,dx \nonumber\\
&= (1+\e^{\bm{\psi}_{\g}(a_{\theta})} \delta)
 \int_{I_+} \exp\big( {-}\psi'_+(a_{\theta})(x-a_{\theta}) +a_{\theta}x -a_{\theta}^2 \big)
 \,\bm{\gamma}(dx) \nonumber\\
&= (1+\e^{\bm{\psi}_{\g}(a_{\theta})} \delta)
 \int_{I_+} \exp\Big( \big( a_{\theta}-\psi'_+(a_{\theta}) \big) (x-a_{\theta}) \Big)
 \,\bm{\gamma}(dx). \label{eq:1-theta}
\end{align}
Since $\bm{\gamma}([a_{\theta},\infty))=1-\theta$ by the choice of $a_{\theta}$,
this estimate shows that $\psi'_+(a_{\theta})-a_{\theta} \le c_1(\theta,\delta)$
with $\lim_{\delta \to 0} c_1(\theta,\delta) =0$.
We similarly observe $\psi'_-(a_{\theta})-a_{\theta} \ge -c_1$, thereby
\[ -c_1 \le \psi'_-(a_{\theta})-a_{\theta} \le \psi'_+(a_{\theta})-a_{\theta} \le c_1. \]

In order to obtain a more precise estimate,
we assume $\alpha:=\psi'_+(a_{\theta}) -a_{\theta} \ge 0$ and deduce from
$\e^{-t} \le 1 -t +(t^2/2)$ for $t \ge 0$ that
\begin{align*}
&\int_{a_{\theta}}^{\infty}
 \exp\big( {-}\alpha (x-a_{\theta}) \big) \,\bm{\gamma}(dx) \\
&\le \int_{a_{\theta}}^{\infty}
 \bigg( 1 -\alpha(x-a_{\theta}) +\frac{\alpha^2}{2}(x-a_{\theta})^2 \bigg) \,\bm{\gamma}(dx) \\
&= \bigg( 1+\alpha a_{\theta} +\frac{\alpha^2}{2} a_{\theta}^2 \bigg) (1-\theta)
 -\int_{a_{\theta}}^{\infty} (\alpha +\alpha^2 a_{\theta}) x \,\bm{\gamma}(dx)
 +\int_{a_{\theta}}^{\infty} \frac{\alpha^2}{2}x^2 \,\bm{\gamma}(dx) \\
&= \bigg( 1+\alpha a_{\theta} +\frac{\alpha^2}{2} a_{\theta}^2 \bigg) (1-\theta)
 -(\alpha+\alpha^2 a_{\theta}) \frac{\e^{-a_{\theta}^2/2}}{\sqrt{2\pi}}
 +\frac{\alpha^2}{2} \bigg( \frac{a_{\theta} \e^{-a_{\theta}^2/2}}{\sqrt{2\pi}} +1-\theta \bigg) \\
&= \bigg( 1+\alpha a_{\theta} +\frac{\alpha^2}{2}(a_{\theta}^2 +1) \bigg) (1-\theta)
 -\bigg( \alpha+\frac{\alpha^2 a_{\theta}}{2} \bigg) \frac{\e^{-a_{\theta}^2/2}}{\sqrt{2\pi}}.
\end{align*}
Substituting this into \eqref{eq:1-theta} yields
\[ (1+\e^{\bm{\psi}_{\g}(a_{\theta})} \delta)
\bigg\{ \bigg( \alpha+\frac{\alpha^2 a_{\theta}}{2} \bigg) \frac{\e^{-a_{\theta}^2/2}}{\sqrt{2\pi}}
 -\bigg( \alpha a_{\theta} +\frac{\alpha^2}{2}(a_{\theta}^2 +1) \bigg) (1-\theta) \bigg\}
 \le (1-\theta) \e^{\bm{\psi}_{\g}(a_{\theta})} \delta. \]
Recalling $\lim_{\delta \to 0}\alpha =0$, we obtain
\begin{equation}\label{eq:alpha}
\bigg( \frac{\e^{-a_{\theta}^2/2}}{\sqrt{2\pi}} -a_{\theta}(1-\theta) \bigg)
 \limsup_{\delta \to 0} \frac{\alpha}{\delta} \le (1-\theta) \sqrt{2\pi} \e^{a_{\theta}^2/2}.
\end{equation}

In the LHS of \eqref{eq:alpha}, we shall show that
\begin{equation}\label{eq:ted}
\frac{\e^{-a_{\theta}^2/2}}{\sqrt{2\pi}} -a_{\theta}(1-\theta)
 =\cI_{\infty}(\theta)-a_{\theta}(1-\theta) >0,
\end{equation}
where we put $\cI_{\infty}:=\cI_{(\R,\bm{\gamma})}$ similarly to the proof of Lemma~\ref{lm:I_D}.
Notice that the claim is clear when $\theta \le 1/2$ ($a_{\theta} \le 0$),
thereby we assume $\theta>1/2$.
Since
\[ \frac{d}{d\theta}\big[ \cI_{\infty}(\theta)-a_{\theta}(1-\theta) \big]
 =\cI'_{\infty}(\theta) +a_{\theta} -\frac{da_{\theta}}{d\theta}(1-\theta)
 =-\frac{da_{\theta}}{d\theta}(1-\theta) <0 \]
by \eqref{eq:I'} (with $K=1$), it suffices to see $\lim_{\theta \to 1} a_{\theta}(1-\theta)=0$.
Observe that
\[ (2\theta-1)^2 =\frac{1}{2\pi} \bigg( \int_{-a_{\theta}}^{a_{\theta}} \e^{-x^2/2} \,dx \bigg)^2
 \ge \int_0^{a_{\theta}} \e^{-r^2/2} r \,dr
 = 1-\e^{-a_{\theta}^2/2}. \]
This yields
\[ 0 \le 4a_{\theta} \theta(1-\theta) \le a_{\theta} \e^{-a_{\theta}^2/2} \,\to\, 0 \]
as $\theta \to 1$ ($a_{\theta} \to \infty$).
Thus we have the claim \eqref{eq:ted}. 
This in particular shows that \eqref{eq:alpha} holds regardless of $\alpha \ge 0$ or not.

It follows from \eqref{eq:alpha} that
\[ \limsup_{\delta \to 0} \frac{\psi'_+(a_{\theta}) -a_{\theta}}{\delta}
 \le \frac{2\pi (1-\theta) \e^{a_{\theta}^2/2}}{\e^{-a_{\theta}^2/2} -\sqrt{2\pi}a_{\theta}(1-\theta)}. \]
One can similarly show that
\[ \liminf_{\delta \to 0} \frac{\psi'_-(a_{\theta}) -a_{\theta}}{\delta}
 \ge -\frac{2\pi \theta \e^{a_{1-\theta}^2/2}}{\e^{-a_{1-\theta}^2/2} -\sqrt{2\pi}a_{1-\theta} \theta}
 =-\frac{2\pi \theta \e^{a_{\theta}^2/2}}{\e^{-a_{\theta}^2/2} +\sqrt{2\pi}a_{\theta} \theta}. \]
Therefore we conclude
\begin{equation}\label{eq:psi'/e}
-C_2(\theta)
 \le \liminf_{\delta \to 0} \frac{\psi'_-(a_{\theta}) -a_{\theta}}{\delta}
 \le \limsup_{\delta \to 0} \frac{\psi'_+(a_{\theta}) -a_{\theta}}{\delta}
 \le C_2(\theta)
\end{equation}
with $C_2(\theta)>0$ depending only on $\theta$.
\end{step}

\begin{step}[Choice of $T$]
In order to fix a range where we estimate $\psi$ from above,
we take $T>a_{\theta}$ such that
\begin{equation}\label{eq:Tdef}
(\e^{-\bm{\psi}_{\g}(a_{\theta})} +\delta)
 \int_{a_{\theta}}^T \exp\bigg( {-}\psi'_+(a_{\theta})(x-a_{\theta}) -\frac{(x-a_{\theta})^2}{2} \bigg) \,dx
 =1-\theta -\sqrt{\delta}.
\end{equation}
Recall from \eqref{eq:1-theta} that
\[ (\e^{-\bm{\psi}_{\g}(a_{\theta})} +\delta)
 \int_{I_+} \exp\bigg( {-}\psi'_+(a_{\theta})(x-a_{\theta}) -\frac{(x-a_{\theta})^2}{2} \bigg) \,dx
 \ge 1-\theta, \]
therefore such $T \in I_+$ indeed exists.
Moreover, as $\delta \to 0$, we deduce from $\psi'_+(a_{\theta}) \to a_{\theta}$ that $T \to \infty$.
Quantitatively, we put $C'_2:=C_2(\theta)+1$ and observe from \eqref{eq:Tdef} and \eqref{eq:psi'/e} that,
for sufficiently small $\delta$,
\begin{align*}
1-\theta -\sqrt{\delta}
&\le  (\e^{-\bm{\psi}_{\g}(a_{\theta})} +\delta)
 \int_{a_{\theta}}^T \exp\bigg( {-}a_{\theta}(x-a_{\theta}) -\frac{(x-a_{\theta})^2}{2}
 +C'_2(x-a_{\theta}) \delta \bigg) \,dx \\
&= \frac{1+\e^{\bm{\psi}_{\g}(a_{\theta})} \delta}{\sqrt{2\pi}}
 \int_{a_{\theta}}^T \exp\bigg( {-}\frac{x^2}{2} +C'_2(x-a_{\theta}) \delta \bigg) \,dx \\
&= \frac{1+\e^{\bm{\psi}_{\g}(a_{\theta})} \delta}{\sqrt{2\pi}}
 \int_{a_{\theta}}^T \exp\bigg( {-}\frac{(x-C'_2 \delta)^2}{2}
 -C'_2 a_{\theta} \delta +\frac{(C'_2 \delta)^2}{2} \bigg) \,dx \\
&= (1+\e^{\bm{\psi}_{\g}(a_{\theta})} \delta)
 \exp\bigg( {-}C'_2 a_{\theta} \delta +\frac{(C'_2 \delta)^2}{2} \bigg)
 \bm{\gamma}\big( [a_{\theta}-C'_2 \delta,T-C'_2 \delta] \big).
\end{align*}
Combining this with
\[ \bm{\gamma}\big( [a_{\theta}-C'_2 \delta,T-C'_2 \delta] \big)
 \le \bm{\gamma}\big( [a_{\theta},T] \big) +\frac{C'_2 \delta}{\sqrt{2\pi}}
 = 1-\theta -\bm{\gamma}\big( [T,\infty) \big) +\frac{C'_2 \delta}{\sqrt{2\pi}}, \]
we obtain
\begin{equation}\label{eq:gT}
\bm{\gamma}\big( [T,\infty) \big) \le \sqrt{\delta} +C_3(\theta) \delta.
\end{equation}
Then we also find from
\[ \bm{\gamma}\big( [T,\infty) \big)
 \ge \frac{1}{\sqrt{2\pi}} \int_T^{\infty} \frac{x^2 +x+1}{(x+1)^2}\e^{-x^2/2} \,dx
 =-\frac{1}{\sqrt{2\pi}} \bigg[ \frac{\e^{-x^2/2}}{x+1} \bigg]_T^{\infty}
 =\frac{1}{\sqrt{2\pi}} \frac{\e^{-T^2/2}}{T+1} \]
that
\begin{equation}\label{eq:eT}
\frac{\e^{-T^2/2}}{T+1} \le \sqrt{2\pi\delta} +\sqrt{2\pi}C_3(\theta) \delta.
\end{equation}
\end{step}

\begin{step}[Estimates of $\psi-\bm{\psi}_{\g}$ on $I_+$]
Now we put
\[ \rho(x):=\psi(x) -\bm{\psi}_{\g}(x). \]
Our goal is to bound this difference of weight functions from below and above.
Note first that, by \eqref{eq:psi(x)}, for $x \in I_+$,
\begin{align*}
\psi(x)
&\ge -\log(\e^{-\bm{\psi}_{\g}(a_{\theta})} +\delta) +\psi'_+(a_{\theta})(x-a_{\theta})
 +\frac{(x-a_{\theta})^2}{2} \\
&= \bm{\psi}_{\g}(a_{\theta})  -\log(1+\e^{\bm{\psi}_{\g}(a_{\theta})} \delta)
 +\psi'_+(a_{\theta})(x-a_{\theta}) +\frac{(x-a_{\theta})^2}{2} \\
&= \bm{\psi}_{\g}(x) +\big( \psi'_+(a_{\theta})-a_{\theta} \big)(x-a_{\theta})
 -\log(1+\e^{\bm{\psi}_{\g}(a_{\theta})} \delta),
\end{align*}
thereby
\begin{equation}\label{eq:delta>+}
\rho(x) \ge \big( \psi'_+(a_{\theta})-a_{\theta} \big)(x-a_{\theta})
 -\log(1+\e^{\bm{\psi}_{\g}(a_{\theta})} \delta).
\end{equation}

Next, for $x \in I_+ \cap [T,\infty)$, it follows from the $1$-convexity of $\psi$ that
\begin{align*}
\psi(x) &\ge \psi(T) +\psi'_+(T) (x-T) +\frac{(x-T)^2}{2} \\
&\ge \bm{\psi}_{\g}(T) +\rho(T)
 +\big( \psi'_+(a_{\theta})+(T-a_{\theta}) \big) (x-T) +\frac{(x-T)^2}{2} \\
&= \bm{\psi}_{\g}(x) +\rho(T) +\big( \psi'_+(a_{\theta}) -a_{\theta} \big) (x-T).
\end{align*}
By integration, we have on one hand
\[ \int_{I_+ \cap [T,\infty)} \e^{-\psi} \,dx
 \le \e^{-\rho(T)} \int_{I_+ \cap [T,\infty)}
 \exp\Big( {-}\big( \psi'_+(a_{\theta}) -a_{\theta} \big)(x-T) \Big) \,\bm{\gamma}(dx), \]
and
\begin{align*}
&\int_T^{\infty}
 \exp\Big( {-}\big( \psi'_+(a_{\theta}) -a_{\theta} \big)(x-T) \Big) \,\bm{\gamma}(dx) \\
&=  \exp\bigg( \big( \psi'_+(a_{\theta}) -a_{\theta})T
 +\frac{(\psi'_+(a_{\theta})-a_{\theta})^2}{2} \bigg) \\
&\quad \times \frac{1}{\sqrt{2\pi}} \int_T^{\infty}
 \exp\bigg( {-}\frac{(x+\psi'_+(a_{\theta})-a_{\theta})^2}{2} \bigg) \,dx \\
&= \exp\bigg( \big( \psi'_+(a_{\theta}) -a_{\theta})T
 +\frac{(\psi'_+(a_{\theta})-a_{\theta})^2}{2} \bigg)
 \bm{\gamma} \big( [ T+\psi'_+(a)-a_{\theta},\infty) \big).
\end{align*}
On the other hand, \eqref{eq:a_theta}, \eqref{eq:psi(x)} and \eqref{eq:Tdef} yield
\begin{align*}
&\int_{I_+ \cap [T,\infty)} \e^{-\psi} \,dx
 =(1-\theta) -\int_{a_{\theta}}^T \e^{-\psi} \,dx \\
&\ge (1-\theta) -(\e^{-\bm{\psi}_{\g}(a_{\theta})} +\delta)
 \int_{a_{\theta}}^T \exp\bigg( {-}\psi'_+(a_{\theta})(x-a_{\theta}) -\frac{(x-a_{\theta})^2}{2} \bigg) \,dx \\
&= \sqrt{\delta}.
\end{align*}
Combining these we obtain
\[ \rho(T) \le
 \big( \psi'_+(a_{\theta}) -a_{\theta} \big)T +\frac{(\psi'_+(a_{\theta})-a_{\theta})^2}{2}
 +\log\bigg( \frac{1}{\sqrt{\delta}}
 \bm{\gamma}\big( [T+\psi'_+(a_{\theta}) -a_{\theta},\infty) \big) \bigg). \]
Now, for $x \in [a_{\theta},T]$,
\begin{align*}
\psi(x) &\le \psi(T) -\psi'_+(x)(T-x) -\frac{(T-x)^2}{2} \\
&\le \bm{\psi}_{\g}(T) +\rho(T) -\big( \psi'_+(a_{\theta})+(x-a_{\theta}) \big) (T-x) -\frac{(T-x)^2}{2} \\
&= \bm{\psi}_{\g}(x) +\rho(T) -\big( \psi'_+(a_{\theta})-a_{\theta} \big) (T-x).
\end{align*}
Therefore we conclude, for $x \in [a_{\theta},T]$,
\begin{align}
&\rho(x)
 \le \rho(T) -\big( \psi'_+(a_{\theta})-a_{\theta} \big) (T-x) \nonumber\\
&\le \big( \psi'_+(a_{\theta})-a_{\theta} \big)x +\frac{(\psi'_+(a_{\theta})-a_{\theta})^2}{2}
  +\log\bigg( \frac{1}{\sqrt{\delta}}
 \bm{\gamma}\big( [T+\psi'_+(a_{\theta})-a_{\theta},\infty) \big) \bigg). \label{eq:delta<+}
\end{align}
Thus we could estimate $\rho$ from below \eqref{eq:delta>+} on $I_+$
and from above \eqref{eq:delta<+} on $[a_{\theta},T]$.
\end{step}

\begin{step}[Further estimate]
Notice that the second term in the last line of \eqref{eq:delta<+}
is bounded above by \eqref{eq:psi'/e}.
In order to understand the behavior of the third term as $\delta \to 0$,
we separately discuss the cases of $\psi'_+(a_{\theta}) \ge a_{\theta}$
and $\psi'_+(a_{\theta})<a_{\theta}$.
In the easier case of $\psi'_+(a_{\theta}) \ge a_{\theta}$,
we deduce from \eqref{eq:Tdef} that
\begin{align*}
&\bm{\gamma}\big( [T+\psi'_+(a_{\theta})-a_{\theta},\infty) \big)
 \le \bm{\gamma}\big( [T,\infty) \big) \\
&\le (1-\theta) -\frac{1}{\sqrt{2\pi}}
 \int_{a_{\theta}}^T \exp\bigg( {-}\frac{x^2}{2}
 -\big( \psi'_+(a_{\theta}) -a_{\theta} \big) (x-a_{\theta}) \bigg) \,dx \\
&= (1-\theta) -\frac{1}{\sqrt{2\pi}}
 \frac{1-\theta-\sqrt{\delta}}{\e^{-\bm{\psi}_{\g}(a_{\theta})}+\delta} \e^{-a_{\theta}^2/2} \\
&= \frac{\sqrt{\delta} +\sqrt{2\pi} (1-\theta) \e^{a_{\theta}^2/2} \delta}
{1+\sqrt{2\pi} \e^{a_{\theta}^2/2} \delta}.
\end{align*}
Since $\log(1+t) \le t$ for $t \ge 0$, we obtain
\begin{equation}\label{eq:e-to-0+}
\log\bigg( \frac{1}{\sqrt{\delta}} \bm{\gamma}\big( [T+\psi'_+(a_{\theta})-a_{\theta},\infty) \big) \bigg)
 \le \sqrt{2\pi} (1-\theta) \e^{a_{\theta}^2/2} \sqrt{\delta}.
\end{equation}

If $\psi'_+(a_{\theta})<a_{\theta}$, then we need a sharper estimate via \eqref{eq:psi'/e}.
Let us begin with
\[ \bm{\gamma}\big( [T+\psi'_+(a_{\theta})-a_{\theta},\infty) \big)
 = (1-\theta) -\frac{1}{\sqrt{2\pi}} \int_{a_{\theta}}^{T+\psi'_+(a_{\theta})-a_{\theta}} \e^{-x^2/2} \,dx. \]
By \eqref{eq:Tdef},
\begin{align*}
&\int_{a_{\theta}}^{T+\psi'_+(a_{\theta})-a_{\theta}} \e^{-x^2/2} \,dx\\
&= \int_{2a_{\theta} -\psi'_+(a_{\theta})}^T
 \exp\bigg( {-}\frac{(x+\psi'_+(a_{\theta})-a_{\theta})^2}{2} \bigg) \,dx \\
&\ge \int_{a_{\theta}}^T
 \exp\bigg( {-}\frac{(x+\psi'_+(a_{\theta})-a_{\theta})^2}{2} \bigg) \,dx
 -\big( a_{\theta} -\psi'_+(a_{\theta}) \big) \\
&= \e^{-\psi'_+(a_{\theta})^2/2}
 \int_{a_{\theta}}^T \exp\bigg( {-}\psi'_+(a_{\theta})(x-a_{\theta}) -\frac{(x-a_{\theta})^2}{2} \bigg) \,dx
 -\big( a_{\theta} -\psi'_+(a_{\theta}) \big) \\
&= \e^{-\psi'_+(a_{\theta})^2/2}
 \frac{1-\theta-\sqrt{\delta}}{\e^{-\bm{\psi}_{\g}(a_{\theta})}+\delta}
 -\big( a_{\theta} -\psi'_+(a_{\theta}) \big).
\end{align*}
Hence we have
\begin{align*}
&\frac{1}{\sqrt{\delta}} \bm{\gamma}\big( [T+\psi'_+(a_{\theta})-a_{\theta},\infty) \big) \\
&\le \frac{1-\theta}{\sqrt{\delta}} -\e^{-\psi'_+(a_{\theta})^2/2}
 \frac{(1-\theta) \delta^{-1/2} -1}{\e^{-a_{\theta}^2/2} +\sqrt{2\pi}\delta}
 +\frac{a_{\theta} -\psi'_+(a_{\theta})}{\sqrt{2\pi\delta}} \\
&= \frac{1-\theta}{\sqrt{\delta}} \bigg(
 1-\frac{\e^{-\psi'_+(a_{\theta})^2/2}}{\e^{-a_{\theta}^2/2} +\sqrt{2\pi}\delta} \bigg)
 +\frac{\e^{-\psi'_+(a_{\theta})^2/2}}{\e^{-a_{\theta}^2/2} +\sqrt{2\pi}\delta}
 +\frac{a_{\theta} -\psi'_+(a_{\theta})}{\sqrt{2\pi\delta}}.
\end{align*}
Then we deduce from \eqref{eq:psi'/e} and $|(\e^{-t^2/2})'| \le \e^{-1/2}<1$ for $t \in \R$ that,
for sufficiently small $\delta$ and $C'_2:=C_2(\theta) +1$,
\begin{align*}
&\frac{1}{\sqrt{\delta}} \bm{\gamma}\big( [T+\psi'_+(a_{\theta})-a_{\theta},\infty) \big) \\
&\le \frac{1-\theta}{\sqrt{\delta}}
 \frac{C'_2\delta +\sqrt{2\pi}\delta}{\e^{-a_{\theta}^2/2} +\sqrt{2\pi}\delta}
 +\frac{\e^{-a_{\theta}^2/2} +C'_2 \delta}{\e^{-a_{\theta}^2/2} +\sqrt{2\pi}\delta}
 +\frac{C'_2}{\sqrt{2\pi}} \sqrt{\delta} \\
&\le \frac{\e^{-a_{\theta}^2/2} +C_4(\theta) \sqrt{\delta}}{\e^{-a_{\theta}^2/2} +\sqrt{2\pi}\delta}.
\end{align*}
Therefore the same argument as \eqref{eq:e-to-0+} shows
\begin{equation}\label{eq:e-to-0+'}
\log\bigg( \frac{1}{\sqrt{\delta}} \bm{\gamma}\big( [T+\psi'_+(a_{\theta})-a_{\theta},\infty) \big) \bigg)
 \le C_4(\theta) \e^{a_{\theta}^2/2} \sqrt{\delta}.
\end{equation}
\end{step}

\begin{step}[Estimates on $I_-$]
For $x \in I_-$ we can apply similar calculations, however,
we need an additional care to replace $\psi'_-(a_{\theta})$ with $\psi'_+(a_{\theta})$.
We have for $x \in I_-$ ($x \le a_{\theta}$) the analogue to \eqref{eq:psi(x)},
\begin{align*}
\psi(x)
&\ge \psi(a_{\theta}) +\psi'_-(a_{\theta})(x-a_{\theta}) +\frac{(x-a_{\theta})^2}{2} \\
&\ge -\log(\e^{-\bm{\psi}_{\g}(a_{\theta})}+\delta)
 +\psi'_+(a_{\theta})(x-a_{\theta}) +\frac{(x-a_{\theta})^2}{2},
\end{align*}
by the $1$-convexity of $\psi$, \eqref{eq:deficit} and $\psi'_-(a_{\theta}) \le \psi'_+(a_{\theta})$.
This implies
\begin{equation}\label{eq:delta>-}
\rho(x) \ge \big( \psi'_+(a_{\theta}) -a_{\theta} \big) (x-a_{\theta})
 -\log(1+\e^{\bm{\psi}_{\g}(a_{\theta})} \delta)
\end{equation}
on $I_-$ in the same way as \eqref{eq:delta>+}.

In order to have an estimate from above, take $S<a_{\theta}$ such that
\begin{equation}\label{eq:Sdef}
( \e^{-\bm{\psi}_{\g}(a_{\theta})} +\delta)
 \int_S^{a_{\theta}} \exp\bigg( {-}\psi'_+(a_{\theta})(x-a_{\theta})
 -\frac{(x-a_{\theta})^2}{2} \bigg) \,dx =\theta-\sqrt{\delta}.
\end{equation}
Since
\[ ( \e^{-\bm{\psi}_{\g}(a_{\theta})} +\delta)
 \int_{I_-} \exp\bigg( {-}\psi'_+(a_{\theta})(x-a_{\theta})
 -\frac{(x-a_{\theta})^2}{2} \bigg) \,dx \ge \fm(I_-) =\theta \]
similarly to \eqref{eq:1-theta}, we indeed can find $S \in I_-$.
Notice also that $S \to -\infty$ as $\delta \to 0$ and,
similarly to \eqref{eq:gT} and \eqref{eq:eT},
\begin{equation}\label{eq:geS}
\bm{\gamma}\big( (-\infty,S] \big) \le \sqrt{\delta} +C_3(\theta) \delta, \qquad
 \frac{\e^{-S^2/2}}{1-S} \le \sqrt{2\pi\delta} +\sqrt{2\pi}C_3(\theta) \delta
\end{equation}
hold for sufficiently small $\delta$ (by replacing $C_3$ if necessary).

For $x \in I_- \cap (-\infty,S]$, we have
\begin{align*}
\psi(x) &\ge \psi(S) +\psi'_-(S) (x-S) +\frac{(x-S)^2}{2} \\
&\ge \bm{\psi}_{\g}(S) +\rho(S)
 +\big( \psi'_+(a_{\theta})+(S-a_{\theta}) \big) (x-S) +\frac{(x-S)^2}{2} \\
&= \bm{\psi}_{\g}(x) +\rho(S) +\big( \psi'_+(a_{\theta})-a_{\theta} \big)(x-S).
\end{align*}
By integration we deduce that
\[ \int_{I_- \cap (-\infty,S]} \e^{-\psi} \,dx
 \le \e^{-\rho(S)} \int_{I_- \cap (-\infty,S]}
 \exp\Big( {-}\big( \psi'_+(a_{\theta})-a_{\theta} \big) (x-S) \Big) \,\bm{\gamma}(dx). \]
We also observe
\begin{align*}
&\int_{-\infty}^S
 \exp\Big( {-}\big( \psi'_+(a_{\theta})-a_{\theta} \big)(x-S) \Big) \,\bm{\gamma}(dx) \\
&= \exp\bigg( \big( \psi'_+(a_{\theta})-a_{\theta} \big) S
 +\frac{(\psi'_+(a_{\theta})-a_{\theta})^2}{2} \bigg) \\
&\quad \times \frac{1}{\sqrt{2\pi}} \int_{-\infty}^S
 \exp\bigg( {-}\frac{(x+\psi'_+(a_{\theta})-a_{\theta})^2}{2} \bigg) \,dx \\
&= \exp\bigg( \big( \psi'_+(a_{\theta})-a_{\theta} \big) S
 +\frac{(\psi'_+(a_{\theta})-a_{\theta})^2}{2} \bigg)
 \bm{\gamma} \big( (-\infty,S+\psi'_+(a_{\theta})-a_{\theta}] \big).
\end{align*}
Combining this with
\begin{align*}
&\int_{I_- \cap (-\infty,S]} \e^{-\psi} \,dx
 =\theta -\int_S^{a_{\theta}} \e^{-\psi} \,dx \\
&\ge \theta -(\e^{-\bm{\psi}_{\g}(a_{\theta})} +\delta)
 \int_S^{a_{\theta}} \exp\bigg( {-}\psi'_+(a_{\theta})(x-a_{\theta}) -\frac{(x-a_{\theta})^2}{2} \bigg) \,dx \\
&= \sqrt{\delta},
\end{align*}
we obtain
\[ \rho(S)
 \le \big( \psi'_+(a_{\theta})-a_{\theta} \big) S +\frac{(\psi'_+(a_{\theta})-a_{\theta})^2}{2}
 +\log \bigg( \frac{1}{\sqrt{\delta}} \bm{\gamma}
 \big( (-\infty,S+\psi'_+(a_{\theta})-a_{\theta}] \big) \bigg). \]
Then, for $x \in [S,a_{\theta}]$, we have
\begin{align*}
\psi(x) &\le \psi(S) -\psi'_-(x)(S-x) -\frac{(S-x)^2}{2} \\
&\le \bm{\psi}_{\g}(S) +\rho(S) -\big( \psi'_+(a_{\theta})+(x-a_{\theta}) \big) (S-x) -\frac{(S-x)^2}{2} \\
&= \bm{\psi}_{\g}(x) +\rho(S) -\big( \psi'_+(a_{\theta})-a_{\theta} \big)(S-x)
\end{align*}
and hence
\begin{equation}\label{eq:delta<-}
\rho(x)
 \le \big( \psi'_+(a_{\theta})-a_{\theta} \big)x +\frac{(\psi'_+(a_{\theta})-a_{\theta})^2}{2}
 +\log \bigg( \frac{1}{\sqrt{\delta}} \bm{\gamma} \big( (-\infty,S+\psi'_+(a_{\theta})-a_{\theta}] \big) \bigg).
\end{equation}
We also observe, for sufficiently small $\delta$,
\begin{equation}\label{eq:e-to-0-}
 \log \bigg( \frac{1}{\sqrt{\delta}} \bm{\gamma}
 \big( (-\infty,S +\psi'_+(a_{\theta})-a_{\theta}] \big) \bigg)
 \le C_4(1-\theta) \e^{a_{\theta}^2/2} \sqrt{\delta}
\end{equation}
in the same way as \eqref{eq:e-to-0+} and \eqref{eq:e-to-0+'} by separately considering
the cases of $\psi'_+(a_{\theta}) \le a_{\theta}$ and $\psi'_+(a_{\theta})>a_{\theta}$.
\end{step}

\begin{step}[Conclusion]
Let us summarize the outcomes of our estimations to conclude the proof.
Recall $\rho=\psi-\bm{\psi}_{\g}$.
On one hand, we obtain from \eqref{eq:delta>+}, \eqref{eq:delta>-}
and $\log(1+t) \le t$ for $t \ge 0$ that
\[ \psi(x)-\bm{\psi}_{\g}(x)
 \ge \big( \psi'_+(a_{\theta})-a_{\theta} \big) (x-a_{\theta}) -\e^{\bm{\psi}_{\g}(a_{\theta})} \delta \]
on whole $I$, yielding \eqref{eq:delta>}.
On the other hand, combining \eqref{eq:delta<+} with \eqref{eq:psi'/e},
\eqref{eq:e-to-0+} and \eqref{eq:e-to-0+'} for $x \in I_+$,
and \eqref{eq:delta<-} with \eqref{eq:psi'/e} and \eqref{eq:e-to-0-} for $x \in I_-$,
we have
\[ \psi(x)-\bm{\psi}_{\g}(x)
 \le \big( \psi'_+(a_{\theta})-a_{\theta} \big) (x-a_{\theta}) +\omega(\theta) \sqrt{\delta} \]
for sufficiently small $\delta$ and all $x \in [S,T]$.
This is \eqref{eq:delta<} and completes the proof.
$\qedd$
\end{step}
\end{proof}

The estimates \eqref{eq:delta>} and \eqref{eq:delta<} on the weight function
could be compared with \cite[Proposition~A.3]{CMM} which is, thanks to the finite-dimensionality,
in terms of the deficit in the diameter bound
(not directly of the deficit $\delta$ in the isoperimetric profile as above).

We do not know if the order of $\delta$ in Proposition~\ref{pr:psi} is optimal.
Improving the order in each estimate will improve the order of $\delta$ in Theorem~\ref{th:main}.

As a corollary to Proposition~\ref{pr:psi} together with \eqref{eq:psi'/e},
the unique minimizer of $\psi$ is close to that of $\bm{\psi}_{\g}$, namely $0$
(notice that $0 \in I$ indeed holds when $\delta$ is small enough
since $T \to \infty$ and $S \to -\infty$).
This observation is behind the validity of Proposition~\ref{pr:var_X}.

\section{Small deficit implies small symmetric difference}\label{sc:symm}

We continue the analysis on $1$-dimensional spaces with the help of Proposition~\ref{pr:psi},
and the next proposition corresponds to \cite[Proposition~3.1]{CMM} in our setting.
This may be regarded as a quantitative version of Lemma~\ref{lm:Bobkov}.

\begin{proposition}[Small symmetric difference]\label{pr:symm}
Let $I \subset \R$ be a closed interval equipped with a probability measure
$\fm=\e^{-\psi} \,dx$ such that $\psi$ is $1$-convex.
Fix $\theta \in (0,1)$ and assume that, for a Borel set $A \subset I$ with $\fm(A)=\theta$,
\[ \sP(A) \le \e^{-\bm{\psi}_{\g}(a_{\theta})} +\delta \]
holds for sufficiently small $\delta>0$ $($relative to $\theta)$.
Then we have
\begin{equation}\label{eq:symdef}
\min\Big\{ \fm\big( A \,\triangle\, (-\infty,r_{\fm}^-(\theta)] \big),
 \fm\big( A \,\triangle\, [r_{\fm}^+(\theta),\infty) \big) \Big\}
 \le \frac{\sP(A) -\cI_{(I,\fm)}(\theta)}{C_5(\theta,\delta)},
\end{equation}
where $r_{\fm}^-(\theta), r_{\fm}^+(\theta) \in I$ are defined by
\[ \fm \big( I \cap (-\infty,r_{\fm}^-(\theta)] \big) =\fm\big( I \cap [r_{\fm}^+(\theta),\infty) \big) =\theta, \]
and $\lim_{\delta \to 0}C_5(\theta,\delta) =\infty$.
\end{proposition}

\begin{proof}
By reversing and translating $I$ if necessary,
we assume \eqref{eq:a_theta} and $\cI_{(I,\fm)}(\theta)=\e^{-\psi(a_\theta)}$ without loss of generality.
Hence $r_{\fm}^-(\theta) =a_{\theta}$ and the estimates in Proposition~\ref{pr:psi} are available.
Moreover, by \eqref{eq:delta>} and \eqref{eq:delta<},
$r_{\fm}^+(\theta)$ converges to $-a_{\theta}$ as $\delta \to 0$.

Our goal is to show that $A$ is necessarily close to either $I \cap (-\infty,r_{\fm}^-(\theta)]$
or $I \cap [r_{\fm}^+(\theta),\infty)$.
Since $\sP(A)<\infty$, without loss of generality, let $A$ be the union of open intervals
(see, e.g., \cite[Proposition~12.13]{Mag}).
If there is $x \in \del A \cap [S,T]$, then the hypothesis $\sP(A) \le \e^{-\bm{\psi}_{\g}(a_{\theta})} +\delta$
and \eqref{eq:delta<} yield that
\[ \e^{-\bm{\psi}_{\g}(a_{\theta})} +\delta \ge \e^{-\psi(x)}
 \ge \exp\Big( {-}\bm{\psi}_{\g}(x) -\big( \psi'_+(a_{\theta})-a_{\theta} \big)(x-a_{\theta})
 -\omega(\theta) \sqrt{\delta} \Big). \]
Together with \eqref{eq:psi'/e}, we obtain
\[ \bm{\psi}_{\g}(x) \ge \bm{\psi}_{\g}(a_{\theta}) -c(\theta,\delta) \]
with $\lim_{\delta \to 0}c(\theta,\delta)=0$.
On one hand, this implies that $\del A$ cannot appear between
$-|a_{\theta}|+\ve$ and $|a_{\theta}|-\ve$ for some $\ve=\ve(\theta,\delta)>0$
(provided that $a_{\theta} \neq 0$).
On the other hand, if every $x \in \del A$ is far from $\pm a_{\theta}$
(say, $\e^{-\bm{\psi}_{\g}(x)} < \e^{-\bm{\psi}_{\g}(a_{\theta})}/2$),
then $\fm(A)$ is too large (when $A \supset (-|a_{\theta}|,|a_{\theta}|)$)
or too small (when $A \cap (-|a_{\theta}|,|a_{\theta}|) =\emptyset$).
This latter argument is valid also for $a_{\theta}=0$.
Therefore $\del A$ appears exactly once near either $a_{\theta}$ or $-a_{\theta}$,
and all the other points of $\del A$ are far from $\pm a_{\theta}$.

Since the proofs are common, we will assume that $\del A$ appears near $a_{\theta}$
(as the right end of a component) in the sequel.
Concerning a connected component of $A$ whose boundary points are far from $\pm a_{\theta}$,
we can slide it (in $I$) in the direction opposite to  $a_{\theta}$,
with keeping the total mass and hence the symmetric difference with $(-\infty,r^-_{\fm}(\theta)]$,
and decreasing the perimeter.
We eventually modify $A$ into
\[ \big\{ (-\infty,\alpha) \cup \big( \beta,r_{\fm}^-(\theta)+\xi \big) \cup (\zeta,\infty) \big\} \cap I \]
that we again call $A$,
where $\alpha<\beta \ll r_{\fm}^{\pm}(\theta)$, $\xi \in \R$ and $r_{\fm}^{\pm}(\theta) \ll \zeta$.
We regard as $\zeta=\infty$ if $A$ does not include the interval $(\zeta,\infty)$,
and similarly $\alpha=-\infty$ if $(-\infty,\alpha)$ does not exist.
As $\delta \to 0$, we observe from the above discussion (by virtue of Proposition~\ref{pr:psi})
that $\xi \to 0$, $\alpha \to -\infty$, $\beta \to -\infty$ and $\zeta \to \infty$.

\begin{case}
We first assume $\xi \ge 0$.
\end{case}

For simplicity, we regard $\fm$ as a measure on $\R$ in this proof,
namely $\fm((\alpha,\beta))$ will mean $\fm((\alpha,\beta) \cap I)$.
If $\beta \le \inf I$, then by $\fm(A)=\theta$ we have $A=(-\infty,r_{\fm}^-(\theta)) \cap I$
and there is nothing to prove.
Hence we assume $\inf I<\beta$.
Since $\fm(A)=\theta=\fm((-\infty,r_{\fm}^-(\theta)])$, we find
\[ \fm\big( (r_{\fm}^-(\theta), r_{\fm}^-(\theta)+\xi) \cup (\zeta,\infty) \big)
 =\fm\big( (\alpha,\beta) \big). \]
Thus the symmetric difference between $A$ and $(-\infty,r_{\fm}^-(\theta)]$ satisfies
\begin{equation}\label{eq:AvsR}
\fm\big( A \,\triangle\, (-\infty,r_{\fm}^-(\theta)] \big) =2\fm\big( (\alpha,\beta) \big)
 \le 2\fm\big( (-\infty,\beta) \big).
\end{equation}
It follows from the $1$-convexity of $\psi$ that, for $x<\beta$,
\begin{align*}
\psi(x)
&\ge \psi(\beta) +\psi'_-(\beta)(x-\beta) +\frac{(x-\beta)^2}{2} \\
&\ge \psi(\beta) +\big( \psi'_-(a_{\theta}) -(a_{\theta}-\beta) \big) (x-\beta) +\frac{(x-\beta)^2}{2}.
\end{align*}
Put $\bar{\beta}:=\beta+\psi'_-(a_{\theta})-a_{\theta}$ for brevity.
Then we observe
\begin{align*}
\fm\big( (-\infty,\beta) \big)
&\le \e^{-\psi(\beta)}
 \int_{-\infty}^{\beta} \exp\bigg( {-}\frac{(x-\beta)^2}{2} -\bar{\beta}(x-\beta) \bigg) \,dx \\
&= \e^{-\psi(\beta)} \int_{-\infty}^0 \exp\bigg( {-}\frac{x^2}{2} -\bar{\beta}x \bigg) \,dx,
\end{align*}
and
\[ \int_{-\infty}^0 \exp\bigg( {-}\frac{x^2}{2} -\bar{\beta}x \bigg) \,dx
 =\int_{-\infty}^0 \exp\bigg( {-}\frac{(x+\bar{\beta})^2}{2} +\frac{\bar{\beta}^2}{2} \bigg) \,dx
 =\e^{\bar{\beta}^2/2} \int_{-\infty}^{\bar{\beta}} \e^{-x^2/2} \,dx. \]
Hence
\begin{equation}\label{eq:lHop}
\frac{\fm((-\infty,\beta))}{\e^{-\psi(\beta)}}
 \le \e^{\bar{\beta}^2/2} \int_{-\infty}^{\bar{\beta}}\e^{-x^2/2} \,dx
 \to 0
\end{equation}
as $\delta \to 0$, because by l'H\^{o}pital's rule
\[  \lim_{b \to -\infty} \frac{\int_{-\infty}^b \e^{-x^2/2} \,dx}{\e^{-b^2/2}}
 =\lim_{b \to -\infty} \frac{1}{-b} =0. \]

Now, if $\psi(r_{\fm}^-(\theta)) \ge \psi(r_{\fm}^-(\theta)+\xi)$,
then we immediately obtain
\[ \sP(A) \ge \e^{-\psi(\beta)} +\e^{-\psi(r_{\fm}^-(\theta)+\xi)}
 \ge \e^{-\psi(\beta)} +\e^{-\psi(r_{\fm}^-(\theta))}. \]
Therefore \eqref{eq:AvsR} and \eqref{eq:lHop} show
\[ \frac{\sP(A) -\cI_{(I,\fm)}(\theta)}{\fm(A \,\triangle\, (-\infty,r_{\fm}^-(\theta)])}
 \ge \frac{\e^{-\psi(\beta)}}{2\fm((-\infty,\beta))} \to \infty \]
as $\delta \to 0$.
In the other case of $\psi(r_{\fm}^-(\theta)) <\psi(r_{\fm}^-(\theta)+\xi)$,
let us consider
\[ A':= \big\{ (-\infty,r_{\fm}^-(\theta)+\xi) \cup (\zeta,\infty) \big\} \cap I \]
and put
\[ \theta' :=\fm(A') =\theta +\fm\big( (\alpha,\beta) \big). \]
Note that $r_{\fm}^-(\theta)+\xi \le r_{\fm}^-(\theta')$ and hence
$\psi(r_{\fm}^-(\theta)+\xi) \le \psi(r_{\fm}^-(\theta'))$ by the convexity of $\psi$.
Thus we have
\[ \sP(A) -\cI_{(I,\fm)}(\theta)
 \ge \e^{-\psi(\beta)} +\e^{-\psi(r_{\fm}^-(\theta)+\xi)} -\e^{-\psi(r_{\fm}^-(\theta))}
 \ge \e^{-\psi(\beta)} +\e^{-\psi(r_{\fm}^-(\theta'))} -\e^{-\psi(r_{\fm}^-(\theta))}. \]
Since $[r_{\fm}^-]'(\theta) =\e^{\psi(r_{\fm}^-(\theta))}$ by the definition of $r_{\fm}^-$
(similarly to \eqref{eq:a'}) and $\psi$ is convex, we deduce that
\[ \e^{-\psi(r_{\fm}^-(\theta))} -\e^{-\psi(r_{\fm}^-(\theta'))}
 \le \psi'_- \big( r_{\fm}^-(\theta') \big) (\theta'-\theta)
 \le \psi'_- \big( r_{\fm}^-(\theta') \big) \fm\big( (-\infty,\beta) \big). \]
Therefore
\[ \sP(A) -\cI_{(I,\fm)}(\theta)
 \ge \e^{-\psi(\beta)} -\psi'_- \big( r_{\fm}^-(\theta') \big) \fm\big( (-\infty,\beta) \big), \]
where, thanks to \eqref{eq:lHop}, the RHS is positive if $\delta$ is sufficiently small.
Combining this with \eqref{eq:AvsR}, we conclude
\[ \frac{\sP(A) -\cI_{(I,\fm)}(\theta)}{\fm(A \,\triangle\, (-\infty,r_{\fm}^-(\theta)])}
 \ge \frac{\e^{-\psi(\beta)}}{2\fm((-\infty,\beta))} -\frac{\psi'_- (r_{\fm}^-(\theta'))}{2}
 \to \infty \]
as $\delta \to 0$.

\begin{case}
Next we assume $\xi<0$.
\end{case}

In this case we can discuss similarly by reversing $A$.
Since $\fm((-\infty,r_{\fm}^-(\theta)+\xi))<\theta$, it necessarily holds $\zeta<\infty$.
Then we have
\[ \fm\big( A \,\triangle\, (-\infty,r_{\fm}^-(\theta)] \big) =2\fm\big( (\zeta,\infty) \big) \]
instead of \eqref{eq:AvsR}, and
\begin{equation}\label{eq:lHop'}
\frac{\fm((\zeta,\infty))}{\e^{-\psi(\zeta)}} \to 0
\end{equation}
as $\delta \to 0$ similarly to \eqref{eq:lHop}.
This is enough to conclude if $\psi(r_{\fm}^-(\theta)) \ge \psi(r_{\fm}^-(\theta)+\xi)$, since
\[ \frac{\sP(A) -\cI_{(I,\fm)}(\theta)}{\fm(A \,\triangle\, (-\infty,r_{\fm}^-(\theta)])}
 \ge \frac{\e^{-\psi(\zeta)} +\e^{-\psi(r_{\fm}^-(\theta)+\xi)} -\e^{-\psi(r_{\fm}^-(\theta))}}{2\fm((\zeta,\infty))}
 \ge \frac{\e^{-\psi(\zeta)}}{2\fm((\zeta,\infty))} \to \infty \]
as $\delta \to 0$.
In the case of $\psi(r_{\fm}^-(\theta))<\psi(r_{\fm}^-(\theta)+\xi)$, we consider
\[ A' :=\big\{ (-\infty,\alpha) \cup \big( \beta,r_{\fm}^-(\theta)+\xi \big) \big\} \cap I,
 \qquad \theta' :=\fm(A') =\theta -\fm\big( (\zeta,\infty) \big). \]
Then $r_{\fm}^-(\theta') \le r_{\fm}^-(\theta)+\xi$ and
$\psi(r_{\fm}^-(\theta')) \ge \psi(r_{\fm}^-(\theta)+\xi)$ by the convexity of $\psi$, therefore
\[ \sP(A) -\cI_{(I,\fm)}(\theta)
 \ge \e^{-\psi(\zeta)} +\e^{-\psi(r_{\fm}^-(\theta'))} -\e^{-\psi(r_{\fm}^-(\theta))}
 \ge \e^{-\psi(\zeta)} +\psi'_+ \big( r_{\fm}^-(\theta') \big) \fm\big( (\zeta,\infty) \big). \]
Finally \eqref{eq:lHop'} implies
\[ \frac{\sP(A) -\cI_{(I,\fm)}(\theta)}{\fm(A \,\triangle\, (-\infty,r_{\fm}^-(\theta)])}
 \ge \frac{\e^{-\psi(\zeta)}}{2\fm((\zeta,\infty))} +\frac{\psi'_+ (r_{\fm}^-(\theta'))}{2}
 \to \infty \]
as $\delta \to 0$.

Therefore we conclude, for sufficiently small $\delta$,
\[ \sP(A) -\cI_{(I,\fm)}(\theta)
 \ge C(\theta,\delta) \fm\big( A \,\triangle\, (-\infty,r_{\fm}^-(\theta)] \big) \]
and $\lim_{\delta \to 0}C(\theta,\delta) =\infty$.
When $\del A$ appears near $-a_{\theta}=a_{1-\theta}$,
we similarly obtain
\[ \sP(A) -\cI_{(I,\fm)}(\theta)
 \ge C(\theta,\delta) \fm\big( A \,\triangle\, [r_{\fm}^+(\theta),\infty) \big) \]
(by using $\e^{-\psi(r_{\fm}^+(\theta))} \ge \cI_{(I,\fm)}(\theta)$).
$\qedd$
\end{proof}

\section{Reverse Poincar\'e inequality on needles}\label{sc:revP}

In this section,
we analyze the spectral gap of a $1$-dimensional space with a small isoperimetric deficit.
We shall see that affine functions achieve the sharp spectral gap asymptotically
as the deficit goes to $0$.
Precisely, we show the following reverse form of the Poincar\'e inequality, where
\[ \Var_{(I,\fm)}(u):=\int_I u^2 \,d\fm -\bigg( \int_I u \,d\fm \bigg)^2 \]
is the \emph{variance} of $u$ (recall \eqref{eq:Poin}).

\begin{proposition}[Reverse Poincar\'e inequality on needles]\label{pr:revP}
Let $I \subset \R$ be a closed interval equipped with a probability measure
$\fm=\e^{-\psi} \,dx$ such that $\psi$ is $1$-convex.
Fix $\theta \in (0,1)$ and assume \eqref{eq:a_theta} and
$\e^{-\psi(a_{\theta})} \le \e^{-\bm{\psi}_{\g}(a_{\theta})} +\delta$.
Then, given $\ve \in (0,1)$, if $\delta>0$ is sufficiently small $($relative to $\theta$ and $\ve)$,
we have
\begin{equation}\label{eq:revP}
\Var_{(I,\fm)}(u) \ge \frac{1}{\Lambda(\theta,\ve,\delta)} \int_I |u'|^2 \,d\fm
\end{equation}
for every affine function $u(x)=ax+b$ with $a,b \in \R$,
where $\Lambda(\theta,\ve,\delta) \le (1-C_6(\theta,\ve)\delta^{(1-\ve)/2})^{-1}$
and, in particular, $\lim_{\delta \to 0}\Lambda(\theta,\ve,\delta)=1$.
\end{proposition}

Precisely, the assumption is read as $\cI_{(I,\fm)}(\theta) \le \cI_{(\R,\bm{\gamma})}(\theta)+\delta$
(up to reversing and translating $I$).
Recall that the $1$-convexity of $\psi$ ($\Ric_{\infty} \ge 1$ or $\CD(1,\infty)$)
implies the Poincar\'e inequality
\[ \Var_{(I,\fm)}(u) \le \int_I |u'|^2 \,d\fm. \]
Hence $\Lambda(\theta,\ve,\delta) \ge 1$ necessarily holds.
Note also that we obtain from \eqref{eq:revP} an upper bound of the first nonzero eigenvalue $\lambda_1$
of $-\Delta_{\fm}$ (recall Subsection~\ref{ssc:wRic}):
\[ 1 \le \lambda_1 \le \Lambda(\theta,\ve,\delta). \]

\begin{proof}
We remark that the inequality \eqref{eq:revP} is invariant under affine transformations of $u$,
thereby it suffices to show \eqref{eq:revP} for some $a,b$ with $a \neq 0$.
Thus, let $u(x) =x +\psi'_+(a_{\theta}) -a_{\theta}$ without loss of generality.
First, it clearly holds
\begin{equation}\label{eq:e^p}
\int_I |u'|^2 \,d\fm =\int_I \e^{-\psi} \,dx =1.
\end{equation}

Second, we deduce from \eqref{eq:delta<} and
\begin{equation}\label{eq:delta+}
\frac{x^2}{2} +\big( \psi'_+(a_{\theta})-a_{\theta} \big) (x-a_{\theta})
 =\frac{(x +\psi'_+(a_{\theta})-a_{\theta})^2}{2}
 -\frac{\psi'_+(a_{\theta})^2-a_{\theta}^2}{2}
\end{equation}
that
\[ \int_I u^2 \,d\fm
 \ge \frac{1}{\sqrt{2\pi}} \exp\bigg( \frac{\psi'_+(a_{\theta})^2-a_{\theta}^2}{2} -\omega\sqrt{\delta} \bigg)
 \int_S^T u(x)^2 \exp\bigg( {-}\frac{(x+\psi'_+(a_{\theta})-a_{\theta})^2}{2} \bigg) \,dx. \]
Note that
\[ \exp\bigg( \frac{\psi'_+(a_{\theta})^2-a_{\theta}^2}{2} -\omega\sqrt{\delta} \bigg)
 \ge 1-C(\theta) \sqrt{\delta} \]
by \eqref{eq:psi'/e}, and
\begin{align*}
&\int_S^T u(x)^2 \exp\bigg( {-}\frac{(x+\psi'_+(a_{\theta})-a_{\theta})^2}{2} \bigg) \,dx \\
&= \bigg[ {-}\big( x+\psi'_+(a_{\theta})-a_{\theta} \big)
 \exp\bigg( {-}\frac{(x+\psi'_+(a_{\theta})-a_{\theta})^2}{2} \bigg) \bigg]_S^T \\
&\quad +\int_S^T
 \exp\bigg( {-}\frac{( x+\psi'_+(a_{\theta})-a_{\theta})^2}{2} \bigg) \,dx \\
&= -\bigg[ \big( x+\psi'_+(a_{\theta})-a_{\theta} \big)
 \exp\bigg( {-}\frac{(x+\psi'_+(a_{\theta})-a_{\theta})^2}{2} \bigg) \bigg]_S^T \\
&\quad +\e^{-\psi'_+(a_{\theta})^2/2} \int_S^T
 \exp\bigg( {-}\psi'_+(a_{\theta})(x-a_{\theta}) -\frac{(x-a_{\theta})^2}{2} \bigg) \,dx.
\end{align*}
In the former term, we observe from $|(t \e^{-t^2/2})'| \le 1$, \eqref{eq:psi'/e} and \eqref{eq:eT} that
\begin{align*}
&\big( T+\psi'_+(a_{\theta})-a_{\theta} \big)
 \exp\bigg( {-}\frac{(T+\psi'_+(a_{\theta})-a_{\theta})^2}{2} \bigg)
 \le T \e^{-T^2/2} +|\psi'_+(a_{\theta})-a_{\theta}| \\
&\le T(T+1)^{1-\ve} \e^{-\ve T^2/2} \bigg( \frac{\e^{-T^2/2}}{T+1} \bigg)^{1-\ve} +(C_2 +1)\delta
 \le C(\theta,\ve) \delta^{(1-\ve)/2}.
\end{align*}
We similarly obtain from \eqref{eq:geS} that
\[ \big( S+\psi'_+(a_{\theta})-a_{\theta} \big)
 \exp\bigg( {-}\frac{(S+\psi'_+(a_{\theta})-a_{\theta})^2}{2} \bigg)
 \ge -C(\theta,\ve) \delta^{(1-\ve)/2}. \]
Thanks to \eqref{eq:Tdef} and \eqref{eq:Sdef}, the latter term coincides with
\[ \e^{-\psi'_+(a_{\theta})^2/2} \frac{1-2\sqrt{\delta}}{\e^{-\bm{\psi}_{\g}(a_{\theta})}+\delta}
 =\e^{(a_{\theta}^2 -\psi'_+(a_{\theta})^2)/2}
 \frac{\sqrt{2\pi}(1-2\sqrt{\delta})}{1+\sqrt{2\pi} \e^{a_{\theta}^2/2} \delta}
 \ge \sqrt{2\pi} -C(\theta)\sqrt{\delta}. \]
Hence we have
\begin{equation}\label{eq:x^2e^p}
\int_I u^2 \,d\fm \ge 1-C(\theta,\ve) \delta^{(1-\ve)/2}.
\end{equation}

Finally, by \eqref{eq:delta>}, \eqref{eq:delta<} and \eqref{eq:delta+},
\begin{align*}
&\sqrt{2\pi} \int_I u \,d\fm \\
&\le \exp\bigg( \frac{\psi'_+(a_{\theta})^2-a_{\theta}^2}{2} +\omega\delta \bigg)
 \int_{a_{\theta}-\psi'_+(a_{\theta})}^{\infty} u(x)
 \exp\bigg( {-}\frac{(x+\psi'_+(a_{\theta})-a_{\theta})^2}{2} \bigg) \,dx \\
&\quad +\exp\bigg( \frac{\psi'_+(a_{\theta})^2-a_{\theta}^2}{2} -\omega\sqrt{\delta} \bigg)
 \int_S^{a_{\theta}-\psi'_+(a_{\theta})} u(x)
 \exp\bigg( {-}\frac{(x+\psi'_+(a_{\theta})-a_{\theta})^2}{2} \bigg) \,dx \\
&= -\exp\bigg( \frac{\psi'_+(a_{\theta})^2-a_{\theta}^2}{2} +\omega\delta \bigg)
 \bigg[ \exp\bigg( {-}\frac{(x+\psi'_+(a_{\theta})-a_{\theta})^2}{2} \bigg)
 \bigg]_{a_{\theta}-\psi'_+(a_{\theta})}^{\infty} \\
&\quad -\exp\bigg( \frac{\psi'_+(a_{\theta})^2-a_{\theta}^2}{2} -\omega\sqrt{\delta} \bigg)
 \bigg[  \exp\bigg( {-}\frac{(x+\psi'_+(a_{\theta})-a_{\theta})^2}{2} \bigg)
 \bigg]_S^{a_{\theta}-\psi'_+(a_{\theta})} \\
&= \big( \e^{\omega\delta} -\e^{-\omega\sqrt{\delta}} \big)
 \exp\bigg( \frac{\psi'_+(a_{\theta})^2-a_{\theta}^2}{2} \bigg) \\
&\quad +\exp\bigg( \frac{\psi'_+(a_{\theta})^2-a_{\theta}^2}{2} -\omega\sqrt{\delta} \bigg)
 \exp\bigg( {-}\frac{(S+\psi'_+(a_{\theta})-a_{\theta})^2}{2} \bigg).
\end{align*}
We similarly find
\begin{align*}
&\sqrt{2\pi} \int_I u \,d\fm \\
&\ge \exp\bigg( \frac{\psi'_+(a_{\theta})^2-a_{\theta}^2}{2} -\omega\sqrt{\delta} \bigg)
 \int_{a_{\theta}-\psi'_+(a_{\theta})}^T u(x)
 \exp\bigg( {-}\frac{(x+\psi'_+(a_{\theta})-a_{\theta})^2}{2} \bigg) \,dx \\
&\quad +\exp\bigg( \frac{\psi'_+(a_{\theta})^2-a_{\theta}^2}{2} +\omega\delta \bigg)
 \int_{-\infty}^{a_{\theta}-\psi'_+(a_{\theta})} u(x)
 \exp\bigg( {-}\frac{(x+\psi'_+(a_{\theta})-a_{\theta})^2}{2} \bigg) \,dx \\
&= \big( \e^{-\omega\sqrt{\delta}} -\e^{\omega\delta} \big)
 \exp\bigg( \frac{\psi'_+(a_{\theta})^2-a_{\theta}^2}{2} \bigg) \\
&\quad -\exp\bigg( \frac{\psi'_+(a_{\theta})^2-a_{\theta}^2}{2} -\omega\sqrt{\delta} \bigg)
 \exp\bigg( {-}\frac{(T+\psi'_+(a_{\theta})-a_{\theta})^2}{2} \bigg).
\end{align*}
Therefore we obtain from \eqref{eq:psi'/e}, \eqref{eq:eT}, \eqref{eq:geS}
and $|(\e^{-t^2/2})'| \le \e^{-1/2} <1$ that
\begin{equation}\label{eq:xe^p}
\bigg| \int_I u \,d\fm \bigg| \le C(\theta,\ve) \delta^{(1-\ve)/2}.
\end{equation}

Thanks to \eqref{eq:x^2e^p} and \eqref{eq:xe^p}, we obtain
\[ \Var_{(I,\fm)} (u) \ge 1-C(\theta,\ve) \delta^{(1-\ve)/2}. \]
Combining this with \eqref{eq:e^p} completes the proof.
$\qedd$
\end{proof}

\section{Reverse Poincar\'e inequality on $M$ and applications}\label{sc:rev}

Henceforth we consider Riemannian manifolds and
apply the $1$-dimensional analysis in the previous sections
via the needle decomposition.
This section is devoted to a reverse Poincar\'e inequality on $M$
derived from Proposition~\ref{pr:revP}, followed by several applications.

\subsection{Decomposition of deficit}\label{ssc:peri}

Let $(M,g,\fm)$ be a complete $\cC^{\infty}$-Riemannian manifold
equipped with a measure $\fm=\e^{-\Psi} \vol_g$
such that $\Ric_{\infty} \ge 1$ and $\fm(M)=1$.
Fix $\theta \in (0,1)$ and take a Borel set $A \subset M$ with $\fm(A)=\theta$.

Put $f:=\chi_A -\theta$ and denote by $(Q,\nu)$ and $\{ (X_q,\fm_q) \}_{q \in Q}$
the elements of the needle decomposition as in Theorem~\ref{th:Ndl}.
Then $(X_q,\fm_q)$ enjoys $\Ric_{\infty} \ge 1$ (or $\CD(1,\infty)$)
for $\nu$-almost every $q \in Q$.
Recall from Subsection~\ref{ssc:needle} that this needle decomposition can be used to
prove the isoperimetric inequality $\cI_{(M,\fm)} \ge \cI_{(\R,\bm{\gamma})}$ on $M$
via those on needles $(X_q,\fm_q)$.
We also define $A_q:=A \cap X_q$ for $q \in Q$.
By \cite[Lemma~4.1]{CMM}, one can decompose the isoperimetric deficit of $A$
into those of $A_q$ as follows.

\begin{lemma}[Decomposition of deficit]\label{lm:Lem4.1}
We have
\[ \sP(A)-\cI_{(\R,\bm{\gamma})}(\theta)
 \ge \int_Q \big( \sP(A_q) -\cI_{(\R,\bm{\gamma})}(\theta) \big) \,\nu(dq), \]
where $\sP(A_q)$ denotes the perimeter of $A_q$ in $(X_q,\fm_q)$.
\end{lemma}

We remark that what we need to take care is the measurability of $\sP(A_q)$ in $q \in Q$,
then the inequality itself follows from Fatou's lemma.

\subsection{Reverse Poincar\'e inequality}\label{ssc:revP}

Let $u:M \lra \R$ be the guiding function associated with $f=\chi_A -\theta$ above
(recall Subsection~\ref{ssc:needle}).
Notice that $u \in L^1(\fm)$ holds since $u$ is $1$-Lipschitz
and the measure $\fm$ has the Gaussian decay.
Recall from \eqref{eq:Poin} that we have the Poincar\'e inequality
\[ \Var_{(M,\fm)}(u) \le \int_M |\nabla u|^2 \,d\fm =1, \]
where $|\nabla u|=1$ $\fm$-almost everywhere since $\nu$-almost every needle is not a singleton
(by $f \neq 0$ and Theorem~\ref{th:Ndl}(\ref{Ndl3})).
We shall show a reverse inequality by integrating \eqref{eq:revP} on needles.

\begin{theorem}[Reverse Poincar\'e inequality]\label{th:revP}
Let $(M,g,\fm)$ be a complete weighted Riemannian manifold
such that $\Ric_{\infty} \ge 1$ and $\fm(M)=1$.
Fix $\theta,\ve \in (0,1)$ and take a Borel set $A \subset M$ with $\fm(A)=\theta$
and $\sP(A) \le \cI_{(\R,\bm{\gamma})}(\theta)+\delta$ for sufficiently small $\delta>0$
$($relative to $\theta$ and $\ve)$.
Then the guiding function $u$ associated with $f=\chi_A -\theta$ satisfies
\[ \Var_{(M,\fm)}(u) \ge \frac{1}{\Lambda'(\theta,\ve,\delta)} \int_M |\nabla u|^2 \,d\fm, \]
where $\Lambda'(\theta,\ve,\delta) \le (1-C_7(\theta,\ve)\delta^{(1-\ve)/(3-\ve)})^{-1}$
and in particular $\lim_{\delta \to 0}\Lambda'(\theta,\ve,\delta)=1$.
\end{theorem}

\begin{proof}
We set $a:=(1-\ve)/(3-\ve)$ and consider
\[ Q':=\{ q \in Q \,|\, \fm_q(A_q)=\theta,\,
 \sP(A_q) -\cI_{(\R,\bm{\gamma})}(\theta) <\delta^{1-a} \}. \]
Then we deduce from Lemma~\ref{lm:Lem4.1} that $\nu(Q') \ge 1-\delta^a$.
Precisely, assuming in contrary $\nu(Q') <1-\delta^a$, we have
\[ \delta \ge \int_Q \big( \sP(A_q) -\cI_{(\R,\bm{\gamma})}(\theta) \big) \,\nu(dq)
 \ge \delta^{1-a} \cdot \nu(Q \setminus Q') > \delta,  \]
a contradiction.

For $\nu$-almost every $q \in Q'$, since $u$ is affine and $|u'| \equiv 1$ on $X_q$, \eqref{eq:revP} yields
\[ \Var_{(X_q,\fm_q)}(u)  \ge \frac{1}{\Lambda(\theta,\ve,\delta^{1-a})} \]
for $\Lambda$ from Proposition~\ref{pr:revP}.
Integrating in $q$ implies
\begin{align}
\Var_{(M,\fm)}(u)
&= \int_Q \int_{X_q} u^2 \,d\fm_q \,\nu(dq)
 -\bigg( \int_Q \int_{X_q} u \,d\fm_q \,\nu(dq) \bigg)^2
 \nonumber\\
&\ge \int_Q \bigg\{ \int_{X_q} u^2 \,d\fm_q -\bigg( \int_{X_q} u \,d\fm_q \bigg)^2 \bigg\} \,\nu(dq)
 \nonumber\\
&\ge \frac{1-\delta^a}{\Lambda(\theta,\ve,\delta^{1-a})}, \label{eq:var_X'}
\end{align}
where we used Theorem~\ref{th:Ndl}(\ref{Ndl1}) as well as
the Cauchy--Schwarz inequality on $(Q,\nu)$.
Recalling the choice of $\Lambda$ in Proposition~\ref{pr:revP} and $a=(1-\ve)/(3-\ve)$, we obtain
\[ \frac{\Lambda(\theta,\ve,\delta^{1-a})}{1-\delta^a}
 \le \frac{1}{(1-\delta^a) (1-C_6(\theta,\ve) \delta^{(1-a)(1-\ve)/2})}
 \le \frac{1}{1-C(\theta,\ve)\delta^{(1-\ve)/(3-\ve)}}. \]
This completes the proof.
$\qedd$
\end{proof}

Now let us choose the guiding function $u:M \lra \R$ so that
\[ \int_M u \,d\fm=0 \]
(by replacing $u$ with $u-\int_M u \,d\fm$).
Then, combining \eqref{eq:var_X'} with the Poincar\'e inequality \eqref{eq:Poin}, we obtain
\begin{equation}\label{eq:var_X}
\int_Q \bigg( \int_{X_q} u \,d\fm_q \bigg)^2 \,\nu(dq)
 \le \int_M u^2 \,d\fm -\frac{1}{\Lambda'(\theta,\ve,\delta)}
 \le 1-\frac{1}{\Lambda'(\theta,\ve,\delta)}
 \le C_7(\theta,\ve)\delta^{(1-\ve)/(3-\ve)}.
\end{equation}
Therefore $\int_{X_q} u \,d\fm_q$ is close to $0$ on most needles $q$.

Since most needles are long and the measures on them are close to
the Gaussian measure $\bm{\gamma}$,
\eqref{eq:var_X} shows that, on most needles, the guiding function $u$ attains $0$
at a point close to the maximum of the density function (minimum of the weight function $\psi$).
This observation plays an essential role to integrate the estimates on needles
(see Proposition~\ref{pr:var_X} and the proofs of Proposition~\ref{pr:Pr6.4} and Theorem~\ref{th:main}),
and we stress that the guiding function $u$ is the key ingredient.

\subsection{Reverse logarithmic Sobolev inequality}\label{ssc:rLSI}

Going back to the seminal work of Otto--Villani \cite{OV},
it is now well known that the \emph{logarithmic Sobolev inequality},
\[ \int_M f \log f \,d\fm \le \frac{1}{2\lambda} \int_M \frac{|\nabla f|^2}{f} \,d\fm \]
for nonnegative locally Lipschitz functions $f$ with $\int_M f \,d\fm=1$,
implies the \emph{Talagrand inequality},
\[ W_2(\mu,\fm)^2 \le \frac{2}{\lambda} \Ent_{\fm}(\mu) \]
for $\mu \in \cP^2(M)$,
and the Talagrand inequality implies the Poincar\'e inequality
\[ \Var_{(M,\fm)}(u) \le \frac{1}{\lambda} \int_M |\nabla u|^2 \,d\fm \]
(both without loss of constants).
In the Talagrand inequality, $W_2$ is the \emph{$L^2$-Wasserstein distance},
$\cP^2(M)$ is the set of Borel probability measures on $M$ of finite second moment,
and $\Ent_{\fm}(\mu):=\int_M \rho \log\rho \,d\fm$ with $\mu=\rho\fm$ is the \emph{relative entropy}.
We refer to \cite[Theorem~22.17]{Vi} for a precise statement that is available in our setting,
and to the bibliographical notes in \cite[Chapter~22]{Vi} for a historical account and related results.

By reversing these implications, we deduce from Theorem~\ref{th:revP}
the following reverse forms of logarithmic Sobolev and Talagrand inequalities.

\begin{corollary}[Reverse Talagrand \& log-Sobolev inequalities]\label{cr:rLSI}
Let $(M,g,\fm)$ be as in Theorem~$\ref{th:revP}$
and assume $\cI_{(M,\fm)}(\theta) \le \cI_{(\R,\bm{\gamma})}(\theta) +\delta$
for some $\theta \in (0,1)$ and sufficiently small $\delta>0$.
Then, for any $\lambda >\Lambda'(\theta,\ve,\delta)$, we have the following.
\begin{enumerate}[{\rm (i)}]
\item
There exists some $\mu \in \cP^2(M) \setminus \{\fm\}$ such that
\begin{equation}\label{eq:rT2}
W_2(\mu,\fm)^2 \ge \frac{2}{\lambda} \Ent_{\fm}(\mu).
\end{equation}

\item
There exists some nonconstant, nonnegative, locally Lipschitz function $f$
such that $\int_M f \,d\fm=1$ and
\begin{equation}\label{eq:rLSI}
\int_M f \log f \,d\fm \ge \frac{1}{2\lambda} \int_M \frac{|\nabla f|^2}{f} \,d\fm.
\end{equation}
\end{enumerate}
\end{corollary}

The proofs of the above implications (log-Sobolev to Talagrand, Talagrand to Poincar\'e)
are based on dual formulations and semigroup approaches
(employing heat semigroup \cite{OV} or Hamilton--Jacobi semigroup \cite{BoGL,LVhj}),
and then the relation of $u$ from Theorem~\ref{th:revP},
$\mu$ in \eqref{eq:rT2}, and $f$ in \eqref{eq:rLSI} is seemingly unclear.

In the direct implication from the logarithmic Sobolev inequality to the Poincar\'e inequality
in \cite[Theorem~6.18]{LV}, we have a more explicit argument and can build $f$ from $u$ as follows.
Given any $\lambda>\Lambda'(\theta,\ve,\delta)$, truncating $u$ in Theorem~\ref{th:revP},
we obtain $u_{\sigma}:=\max\{ \min\{u,\sigma\},-\sigma \}$ for some (large) $\sigma>0$
satisfying
\begin{equation}\label{eq:LSI-P}
\Var_{(M,\fm)}(u_{\sigma}) \ge \frac{1}{\lambda} \int_M |\nabla u_{\sigma}|^2 \,d\fm.
\end{equation}
Now let us put $h:=u_{\sigma} -\int_M u_{\sigma} \,d\fm$
and consider the function $f_{\ve}:=1+\ve h$ for $|\ve| < (2\|h\|_{L^{\infty}})^{-1}$.
Note that $f_{\ve} >0$ and $\int_M f_{\ve} \,d\fm=1$.
Then we calculate
\[ \int_M \frac{|\nabla f_{\ve}|^2}{f_{\ve}} \,d\fm
 =\ve^2 \int_M \frac{|\nabla h|^2}{1+\ve h} \,d\fm
 \le \frac{\ve^2}{1-\ve\|h\|_{L^{\infty}}} \int_M |\nabla h|^2 \,d\fm. \]
Moreover, it follows from
\[ (1+t)\log(1+t) \ge t+at^2 \qquad \text{for}\ {-}1<t \le \frac{1}{2a}-1,\ a \in \bigg( 0, \frac{1}{2} \bigg) \]
that
\[ \int_M f_{\ve} \log f_{\ve} \,d\fm
 \ge \int_M \bigg( \ve h +\frac{1}{2(1+\ve\|h\|_{L^{\infty}})} (\ve h)^2 \bigg) \,d\fm
 =\frac{\ve^2}{2(1+\ve\|h\|_{L^{\infty}})} \int_M h^2 \,d\fm. \]
Combining these with \eqref{eq:LSI-P} yields the reverse logarithmic Sobolev inequality
\[ \int_M f_{\ve} \log f_{\ve} \,d\fm
 \ge \frac{1-\ve\|h\|_{L^{\infty}}}{1+\ve\|h\|_{L^{\infty}}} \frac{1}{2\lambda}
 \int_M \frac{|\nabla f_{\ve}|^2}{f_{\ve}} \,d\fm. \]

We close the section with some remarks on related investigations.

\begin{remark}[Related results]\label{rm:quan}
\begin{enumerate}[(a)]
\item
The stability of geometric inequalities on Riemannian manifolds
is an important problem and known to have applications in the study of limit spaces.
For instance, Colding \cite{Co1,Co2} showed that
an $n$-dimensional Riemannian manifold $(M,g)$ satisfying $\Ric_g \ge n-1$
is close to the unit sphere $\Sph^n$ in the Gromov--Hausdorff distance
if and only if the volume $\vol_g(M)$ is close to that of $\Sph^n$.
We will denote the volume of $\Sph^n$ by $\omega_n$.
Notice that $\vol_g(M)$ is not greater than $\omega_n$
by the Bishop comparison theorem (see \cite{Ch}),
and the almost maximal volume implies that the manifold is homeomorphic to $\Sph^n$ by \cite{Per}.
It is also shown in \cite{Co2} that, if the radius of $M$ as above is close to $\pi$,
then its volume is close to $\omega_n$ (thereby $M$ is homeomorphic to $\Sph^n$),
where the \emph{radius} of $M$ is defined as $\inf_{x \in M} \sup_{y \in M}d(x,y)$
and is not greater than $\pi$.
Another result on this kind of `almost sphere theorem' by Petersen \cite{Pet}
asserts that the radius is close to $\pi$
if and only if the $(n+1)$-th eigenvalue of the Laplacian is close to $n$
(later improved to the $n$-th eigenvalue by Aubry \cite{Au}).
We refer to \cite{HM,KaMo} for recent generalizations of some of these results to $\RCD$-spaces
(recall Remark~\ref{rm:wRic}(c)).

\item
Among functional inequalities, the relation between the Poincar\'e inequality
(spectral gap) and the diameter of Riemannian manifolds has been well investigated
(see \cite{BBG,Be,Che,Cr}).
We refer to \cite{CaMoSe} for a recent generalization
to essentially non-branching $\CD(N-1,N)$-spaces ($N \in (1,\infty)$).
In \cite{CaMoSe} they make use of the needle decomposition in the same spirit as
\cite{CMM} on quantitative isoperimetric inequalities.
See also \cite{OT} for the rigidity of the logarithmic Sobolev inequality on weighted Riemannian manifolds
with $\Ric_{\infty} \ge K>0$ (the case of $\Ric_N \ge K>0$ with $N \in [n,\infty)$ is open).
In the Euclidean setting, quantitative estimates in comparison with the Gaussian spaces
are studied in \cite{DF,CF} for the Poincar\'e inequality,
and in \cite{BGRS,FIL,CF} for the logarithmic Sobolev inequality.

\item
We refer to a recent paper \cite[Theorem~2.1]{ABS} for another kind of rigidity result
concerning a gradient estimate on $\RCD(0,N)$-spaces.
\end{enumerate}
\end{remark}

\section{Quantitative isoperimetric inequality}\label{sc:main}

As in the previous section, let $(M,g,\fm)$ be a weighted Riemannian manifold
with $\Ric_{\infty} \ge 1$ and $\fm(M)=1$,
fix $\theta \in (0,1)$ and take a Borel set $A \subset M$ with $\fm(A)=\theta$.
We employ the needle decomposition associated with $f:=\chi_A -\theta$
as in Subsection~\ref{ssc:needle}: $(Q,\nu)$, $\{ (X_q,\fm_q) \}_{q \in Q}$,
and the guiding function $u$ with $\int_M u \,d\fm=0$.
Set $A_q :=A \cap X_q$ as in the previous section.

Put $\delta(A):=\sP(A)-\cI_{(\R,\bm{\gamma})}(\theta)$ and define
\begin{equation}\label{eq:Q_l}
Q_{\ell} :=\big\{ q \in Q \,\big|\, \fm_q(A_q)=\theta, \,
 \sP(A_q) -\cI_{(\R,\bm{\gamma})}(\theta) <\sqrt{\delta(A)} \big\}
\end{equation}
as a set of `long' needles
(recall from Lemma~\ref{lm:I_D} that small deficit implies large diameter).
Notice that $Q_{\ell}$ is a measurable set since the function $q \mapsto \sP(A_q)$ is measurable
by \cite[Lemma~4.1]{CMM}.
We observe from Lemma~\ref{lm:Lem4.1} the following
(similarly to the proof of Theorem~\ref{th:revP}).

\begin{lemma}[$Q_{\ell}$ is large]\label{lm:Q_l}
We have $\nu(Q_{\ell}) \ge 1-\sqrt{\delta(A)}$.
\end{lemma}

For further analyzing the behavior of long needles, we define
\begin{equation}\label{eq:Q^+-}
\begin{array}{l}
Q_{\ell}^- :=\big\{ q \in Q_{\ell} \,\big|\,
 \fm_q \big( A_q \,\triangle\, (-\infty,r_{\fm_q}^-(\theta)] \big) \le \sqrt{\delta(A)} \big\}, \medskip\\
Q_{\ell}^+ :=\big\{ q \in Q_{\ell} \,\big|\,
 \fm_q \big( A_q \,\triangle\, [r_{\fm_q}^+(\theta),\infty) \big) \le \sqrt{\delta(A)} \big\},
\end{array}
\end{equation}
where $X_q$ is parametrized by $u$ and $r_{\fm_q}^{\pm}(\theta) \in X_q$ are defined by
\[ \fm_q \big( X_q \cap (-\infty,r_{\fm_q}^-(\theta)] \big)
 =\fm_q \big( X_q \cap [r_{\fm_q}^+(\theta),\infty) \big) =\theta \]
as in Proposition~\ref{pr:symm}.
The measurability of $Q_{\ell}^+$ and $Q_{\ell}^-$ can be shown as in \cite{CMM}
(see Lemma~6.1 and the paragraph following it).
Then the next lemma is a consequence of Proposition~\ref{pr:symm}.

\begin{lemma}[$Q_{\ell}^- \cup Q_{\ell}^+$ is large]\label{lm:Lem6.2}
If $\delta(A)$ is sufficiently small, then we have
\[ \nu\big( Q_{\ell} \setminus (Q_{\ell}^- \cup Q_{\ell}^+) \big) \le \sqrt{\delta(A)}. \]
\end{lemma}

\begin{proof}
Recall from Theorem~\ref{th:Ndl} that, for $\nu$-almost every $q \in Q$,
$(X_q,\fm_q)$ satisfies $\Ric_{\infty} \ge 1$ and $\fm_q(A_q)=\theta$.
Then we deduce from \eqref{eq:symdef} and
$\lim_{\delta \to 0}C_5(\theta,\delta)=\infty$ that
\[ \sP(A_q) -\cI_{(X_q,\fm_q)}(\theta)
 \ge \min\Big\{ \fm_q \big( A_q \,\triangle\, (-\infty,r_{\fm_q}^-(\theta)] \big),
 \fm_q \big( A_q \,\triangle\, [r_{\fm_q}^+(\theta),\infty) \big) \Big\} \]
for $q \in Q_{\ell}$ provided that $\delta(A)$ is sufficiently small.
Hence
\[ \sP(A_q) -\cI_{(\R,\bm{\gamma})}(\theta) \ge \sP(A_q) -\cI_{(X_q,\fm_q)}(\theta) >\sqrt{\delta(A)} \]
for $q \in Q_{\ell} \setminus (Q_{\ell}^- \cup Q_{\ell}^+)$,
and it follows from Lemma~\ref{lm:Lem4.1} that
\[ \delta(A) \ge \sqrt{\delta(A)} \cdot \nu\big( Q_{\ell} \setminus (Q_{\ell}^- \cup Q_{\ell}^+) \big). \]
$\qedd$
\end{proof}

Next we shall show that one of $Q_{\ell}^-$ and $Q_{\ell}^+$ necessarily has a small volume.
This is the most technical step in this section and the structure of the proof differs from
that of \cite[Proposition~6.4]{CMM},
due to the fact that the diameter of $M$ is not bounded and needles can be infinitely long
(cf., for example, \cite[Proposition~5.1, Corollary~5.4]{CMM}).
The following observation by virtue of \eqref{eq:var_X} will play a crucial role.
Recall that $a_{\theta} \in \R$ is defined by $\bm{\gamma}((-\infty,a_{\theta}])=\theta$.

\begin{proposition}[$u$ is nearly centered on most needles]\label{pr:var_X}
If $\delta(A)$ is sufficiently small, then there exists a measurable set $Q_c \subset Q$
such that $\nu(Q_c) \ge 1-\delta(A)^{(1-\ve)/(9-3\ve)}$ and
\begin{equation}\label{eq:|alpha|}
\max\big\{ |a_{\theta}-r_{\fm_q}^-(\theta)|,|a_{1-\theta}-r_{\fm_q}^+(\theta)| \big\}
 \le C_8(\theta,\ve) \delta(A)^{(1-\ve)/(9-3\ve)}
\end{equation}
for every $q \in Q_c \cap Q_{\ell}$.
\end{proposition}

\begin{proof}
We set $\delta:=\delta(A)$ and
\[ a:=\frac{2(1-\ve)}{3(3-\ve)} \]
for simplicity, and observe from \eqref{eq:var_X} that
the set $Q_c \subset Q$ consisting of $q$ with
\begin{equation}\label{eq:var_Xq}
\bigg( \int_{X_q} u \,d\fm_q \bigg)^2 \le C_7(\theta,\ve) \delta^a
\end{equation}
satisfies $\nu(Q_c) \ge 1-\delta^{(1-\ve)/(3-\ve)-a}$.
Fix a needle $q \in Q_c \cap Q_{\ell}$ and put $\fm_q=\e^{-\psi} \,dx$,
$r^-:=r_{\fm_q}^-(\theta)$ and $r^+:=r_{\fm_q}^+(\theta)$ for brevity.

Since the assertion is symmetric, by reversing $X_q$ if necessary,
we can assume $\cI_{(X_q,\fm_q)}(\theta)=\e^{-\psi(r^-)}$.
Then we have $\e^{-\psi(r^-)} \le \sP(A_q) \le \cI_{(\R,\bm{\gamma})}(\theta) +\sqrt{\delta}$
and deduce from \eqref{eq:delta>} that
\[ \psi(x) -\bm{\psi}_{\g} \big( (x-r^-) +a_{\theta} \big)
 \ge \big( \psi'_+(r^-) -a_{\theta} \big)(x-r^-) -\omega(\theta) \sqrt{\delta} \]
on $X_q$, where we recall that $X_q$ is parametrized by $u$.
We similarly observe from \eqref{eq:delta<} that
\[ \psi(x) -\bm{\psi}_{\g}\big( (x-r^-) +a_{\theta} \big)
 \le \big( \psi'_+(r^-) -a_{\theta} \big)(x-r^-) +\omega(\theta) \delta^{1/4} \]
on $[S+r^- -a_{\theta},T+r^- -a_{\theta}]$.
Let us set $\alpha:=a_{\theta}-r^-$, $\beta:=\psi'_+(r^-) -a_{\theta}$
and observe $|\beta| \le (C_2+1)\sqrt{\delta}$ from \eqref{eq:psi'/e}.
By \eqref{eq:var_Xq} we also find that $\alpha \to 0$ as $\delta \to 0$,
our goal is to make this quantitative.

We have
\begin{align*}
\int_{X_q} u \,d\fm_q
&= \int_{X_q} x \,\fm_q(dx) \\
&\le \int_0^{\infty} x \exp\Big(
 {-}\bm{\psi}_{\g}(x+\alpha) -\beta(x-r^-) +\omega\sqrt{\delta}
 \Big) \,dx \\
&\quad +\int_{S-\alpha}^0 x \exp\Big(
 {-}\bm{\psi}_{\g}(x+\alpha) -\beta(x-r^-) -\omega\delta^{1/4}
 \Big) \,dx \\
&= \frac{1}{\sqrt{2\pi}} \int_0^{\infty} x \exp\bigg(
 {-}\frac{(x +\alpha +\beta)^2}{2} +\alpha\beta
 +\frac{\beta^2}{2} +\beta r^- +\omega\sqrt{\delta} \bigg) \,dx \\
&\quad +\frac{1}{\sqrt{2\pi}} \int_{S-\alpha}^0 x \exp\bigg(
 {-}\frac{(x +\alpha +\beta)^2}{2} +\alpha\beta
 +\frac{\beta^2}{2} +\beta r^- -\omega\delta^{1/4} \bigg) \,dx \\
&= \exp\bigg( \alpha\beta +\frac{\beta^2}{2} +\beta r^- +\omega\sqrt{\delta} \bigg)
 \int_{\alpha +\beta}^{\infty} (x-\alpha-\beta) \,\bm{\gamma}(dx) \\
&\quad +\exp\bigg( \alpha\beta +\frac{\beta^2}{2} +\beta r^- -\omega\delta^{1/4} \bigg)
 \int_{S+\beta}^{\alpha +\beta} (x-\alpha-\beta) \,\bm{\gamma}(dx).
\end{align*}
Since $|\beta| \le (C_2+1)\sqrt{\delta}$ and $\alpha \to 0$ as $\delta \to 0$, we find
\begin{align*}
\exp\bigg( \alpha\beta +\frac{\beta^2}{2} +\beta r^- +\omega\sqrt{\delta} \bigg)
&\le 1+C(\theta)\sqrt{\delta}, \\
\exp\bigg( \alpha\beta +\frac{\beta^2}{2} +\beta r^- -\omega\delta^{1/4} \bigg)
&\ge 1-C(\theta)\delta^{1/4}.
\end{align*}
Moreover, we observe
\begin{align*}
&\int_{S+\beta}^{\alpha+\beta} (x-\alpha-\beta) \,\bm{\gamma}(dx)
 = \int_{-\infty}^{\alpha+\beta} (x-\alpha-\beta) \,\bm{\gamma}(dx)
 -\int_{-\infty}^{S+\beta} (x -\alpha -\beta) \,\bm{\gamma}(dx) \\
&= \int_{-\infty}^{\alpha+\beta} (x-\alpha-\beta) \,\bm{\gamma}(dx)
 +\frac{1}{\sqrt{2\pi}} \big[ \e^{-x^2/2} \big]_{-\infty}^{S+\beta}
 +(\alpha +\beta) \bm{\gamma}\big( (-\infty,S+\beta] \big)
\end{align*}
and, assuming that $\delta$ is sufficiently small,
\[ \e^{-(S+\beta)^2/2}
 \le \e^{-(1-\ve)S^2/2}
 =(1-S)^{(1-\ve)^2} \e^{-\ve(1-\ve)S^2/2} \bigg( \frac{\e^{-S^2/2}}{1-S} \bigg)^{(1-\ve)^2}
 \le C(\theta,\ve) \delta^{(1-\ve)^2/4} \]
and
\[ \bm{\gamma}\big( (-\infty,S+\beta] \big)
 \le \bm{\gamma}\big( (-\infty,S] \big) +\frac{|\beta|}{\sqrt{2\pi}}
 \le C(\theta)\delta^{1/4} \]
by \eqref{eq:geS} and \eqref{eq:Q_l}.
Therefore we obtain
\begin{align*}
\int_{X_q} u \,d\fm_q
&\le \int_{-\infty}^{\infty} (x-\alpha-\beta) \,\bm{\gamma}(dx)
 +C(\theta,\ve)\delta^{(1-\ve)^2/4} \\
&= -\alpha-\beta +C(\theta,\ve)\delta^{(1-\ve)^2/4} \\
&\le -\alpha +C(\theta,\ve)\delta^{(1-\ve)^2/4}.
\end{align*}
A similar calculation shows
\[ \int_{X_q} u \,d\fm_q \ge -\alpha -C(\theta,\ve)\delta^{(1-\ve)^2/4} \]
as well.
Combining these with \eqref{eq:var_Xq} yields (provided that $a/2 \le (1-\ve)^2/4$)
\begin{equation}\label{eq:|alpha|+}
|\alpha| \le \bigg| \alpha +\int_{X_q} u \,d\fm_q \bigg| +\bigg| \int_{X_q} u \,d\fm_q \bigg|
 \le C(\theta,\ve) \delta^{a/2}.
\end{equation}

In order to bound $|a_{1-\theta}-r^+|$, let us recall
\begin{align*}
\e^{-\psi(x)}
&\le \frac{1}{\sqrt{2\pi}}
 \exp\bigg( \alpha\beta +\frac{\beta^2}{2} +\beta r^- +\omega\sqrt{\delta} \bigg)
 \exp\bigg( {-}\frac{(x +\alpha +\beta)^2}{2} \bigg) \\
&\le \big( 1+C(\theta)\sqrt{\delta} \big) \e^{-\bm{\psi}_{\g}(x +\alpha +\beta)}
\end{align*}
on $X_q$.
Therefore, on one hand, for $\Theta >-(\alpha+\beta)$
with $\e^{-\bm{\psi}_{\g}(a_{1-\theta}+\Theta+\alpha+\beta)} \ge \e^{-\bm{\psi}_{\g}(a_{1-\theta})}/2$,
\begin{align*}
\fm_q \big( [a_{1-\theta}+\Theta, \infty) \big)
&\le \big( 1+C(\theta)\sqrt{\delta} \big)
 \bm{\gamma}\big( [a_{1-\theta} +\Theta +\alpha +\beta,\infty) \big) \\
&\le \big( 1+C(\theta)\sqrt{\delta} \big)
 \bigg( \theta -\frac{\e^{-\bm{\psi}_{\g}(a_{1-\theta})}}{2}(\Theta +\alpha +\beta) \bigg).
\end{align*}
Then choosing
\[ \Theta=2\e^{\bm{\psi}_{\g}(a_{1-\theta})} \theta C(\theta) \sqrt{\delta} -\alpha -\beta \]
implies $\fm_q([a_{1-\theta}+\Theta,\infty)) <\theta$ and hence
\[ r^+ <a_{1-\theta} +\Theta \le a_{1-\theta} +C(\theta,\ve) \delta^{a/2}, \]
where we used \eqref{eq:|alpha|+}.
On the other hand, for $\Xi >\alpha +\beta$
with $\e^{-\bm{\psi}_{\g}(a_{1-\theta}-\Xi+\alpha+\beta)} \ge \e^{-\bm{\psi}_{\g}(a_{1-\theta})}/2$,
we observe
\begin{align*}
\fm_q \big( [a_{1-\theta}-\Xi,\infty) \big)
&\ge 1- \big( 1+C(\theta)\sqrt{\delta} \big)
 \bm{\gamma} \big( (-\infty,a_{1-\theta}-\Xi +\alpha +\beta] \big) \\
&\ge 1- \big( 1+C(\theta)\sqrt{\delta} \big)
 \bigg( (1-\theta) -\frac{\e^{-\bm{\psi}_{\g}(a_{1-\theta})}}{2}(\Xi -\alpha -\beta) \bigg).
\end{align*}
This yields $\fm_q([a_{1-\theta}-\Xi,\infty)) >\theta$
with $\Xi=2\e^{\bm{\psi}_{\g}(a_{1-\theta})} (1-\theta) C(\theta) \sqrt{\delta} +\alpha +\beta$,
and hence
\[ r^+ > a_{1-\theta}-\Xi \ge a_{1-\theta} -C(\theta,\ve) \delta^{a/2}. \]
This completes the proof.
$\qedd$
\end{proof}

Let us explain the geometric intuition of the proof of the next proposition.
If both $\nu(Q_{\ell}^-)$ and $\nu(Q_{\ell}^+)$ have a certain volume,
then the strict concavity of $\cI_{(\R,\bm{\gamma})}$ implies that
the sum of the perimeters of regions $A^-$ and $A^+$
corresponding to $Q_{\ell}^-$ and $Q_{\ell}^+$, respectively,
is larger than $\cI_{(\R,\bm{\gamma})}(\theta)$.
This contradicts the assumed small deficit when the gap
between $\sP(A)$ and $\sP(A^-)+\sP(A^+)$ is sufficiently small.
In order to construct such a decomposition of $A$ ($A_{\hat{r}}^-$ and $A_{\hat{r}}^+$ in the proof),
we need an additional assumption $\theta \neq 1/2$.

\begin{proposition}[One of $Q_{\ell}^-$ and $Q_{\ell}^+$ is small]\label{pr:Pr6.4}
Assume $\theta \neq 1/2$.
Then we have
\[ \min\{ \nu(Q_{\ell}^-), \nu(Q_{\ell}^+) \} \le C_9(\theta)\delta(A)^{(1-\ve)/(9-3\ve)}, \]
provided that $\delta(A)$ is sufficiently small.
\end{proposition}

\begin{proof}
Put $\delta=\delta(A)$ again in this proof.
Let us first assume $\theta \in (0,1/2)$ and consider the decomposition of $A$,
\[ A_r^- :=A \cap \{ u \le r \}, \qquad A_r^+ :=A \cap \{ u \ge r \}, \]
for $r \in (r_1,r_2)$ with
\[ r_1:= \frac{2}{3}a_{\theta} +\frac{1}{3}a_{1-\theta}, \qquad
 r_2:= \frac{1}{3}a_{\theta} +\frac{2}{3}a_{1-\theta}. \]
Note that $a_{\theta}<0<a_{1-\theta}=-a_{\theta}$ since $\theta<1/2$.
Moreover, letting $\delta$ smaller if necessary, we find from \eqref{eq:|alpha|} that
$r_1 \ge r^-_{\fm_q}(\theta)$ holds for $q \in Q_c \cap Q_{\ell}$.

Since $|\nabla u|=1$ almost everywhere, we obtain from the coarea formula
(see, e.g., \cite{Ch}) that
\[ \fm \big( A \cap \{r_1 <u< r_2\} \big)
 = \int_{A \,\cap\, \{r_1 <u< r_2 \}} |\nabla u| \,d\fm
 =\int_{r_1}^{r_2} |A \cap u^{-1}(r)| \,dr, \]
where $|\cdot|$ denotes the $(n-1)$-dimensional measure induced from $\fm$
(precisely, $\e^{-\Psi}\mathcal{H}^{n-1}$ where $\mathcal{H}^{n-1}$
is the $(n-1)$-dimensional Hausdorff measure).
For $q \in Q_c \cap Q_{\ell}^-$, we deduce from $r_1 \ge r^-_{\fm_q}(\theta)$ and \eqref{eq:Q^+-} that
\begin{equation}\label{eq:r^-}
\fm_q\big( A_q \cap (-\infty,r_1] \big)
 =\fm_q(A_q) -\fm_q\big( A_q \setminus (-\infty,r_1] \big)
 \ge \theta-\sqrt{\delta}.
\end{equation}
Similarly $\fm_q(A_q \cap [r_2,\infty)) \ge \theta-\sqrt{\delta}$ holds for $q \in Q_c \cap Q_{\ell}^+$.
Then it follows from Theorem~\ref{th:Ndl}\eqref{Ndl1}, Lemmas~\ref{lm:Q_l}, \ref{lm:Lem6.2}
and Proposition~\ref{pr:var_X} that
\begin{align*}
\fm(A^-_{r_1} \cup A^+_{r_2})
&\ge \int_{Q_c \cap Q_{\ell}^-} \fm_q\big( A_q \cap (-\infty,r_1] \big) \,\nu(dq)
 +\int_{Q_c \cap Q_{\ell}^+} \fm_q\big( A_q \cap [r_2,\infty) \big) \,\nu(dq) \\
&\ge (\theta -\sqrt{\delta}) \nu\big( Q_c \cap (Q_{\ell}^- \cup Q_{\ell}^+) \big) \\
&\ge (\theta -\sqrt{\delta}) (1-2\sqrt{\delta}-\delta^{(1-\ve)/(9-3\ve)}) \\
&\ge \theta -(1+2\theta) \sqrt{\delta} -\theta \delta^{(1-\ve)/(9-3\ve)} \\
&\ge \theta -\delta^{(1-\ve)/(9-3\ve)}.
\end{align*}
Therefore we obtain
\[ \int_{r_1}^{r_2} |A \cap u^{-1}(r)| \,dr
 =\theta -\fm(A^-_{r_1} \cup A^+_{r_2}) \le \delta^{(1-\ve)/(9-3\ve)}, \]
and we can choose some $\hat{r} \in (r_1,r_2)$ satisfying
\[ |A \cap u^{-1}(\hat{r})| \le \frac{\delta^{(1-\ve)/(9-3\ve)}}{r_2 -r_1}
 =\frac{3\delta^{(1-\ve)/(9-3\ve)}}{a_{1-\theta} -a_{\theta}}
 =\frac{3\delta^{(1-\ve)/(9-3\ve)}}{2|a_{\theta}|}. \]
This yields that
\begin{equation}\label{eq:r}
\sP(A^-_{\hat{r}}) +\sP(A^+_{\hat{r}}) -\sP(A)
 \le 2|A \cap u^{-1}(\hat{r})| \le \frac{3\delta^{(1-\ve)/(9-3\ve)}}{|a_{\theta}|}.
\end{equation}
In the first inequality, take a sequence $\{\phi_i\}_{i \in \N}$ of Lipschitz functions such that
$0 \le \phi_i \le \chi_A$, $\phi_i \to \chi_A$ in $L^1(\fm)$ and
$\lim_{i \to \infty} \int_M |\nabla \phi_i| \,d\fm=\sP(A)$ (recall \eqref{eq:peri} for the definition of $\sP(A)$),
and put
\[ \rho_i^+(x):=\min\big\{ i \cdot \max\{ u(x)-\hat{r},0 \},1 \big\}, \qquad \rho_i^-(x):=1-\rho_i^+(x). \]
Then $\rho_i^{\pm} \phi_i \to \chi_{A_{\hat{r}}^{\pm}}$ in $L^1(\fm)$ and
\begin{align*}
\sP(A^-_{\hat{r}}) +\sP(A^+_{\hat{r}})
&\le \liminf_{i \to \infty} \int_M \big( |\nabla(\rho_i^- \phi_i)| +|\nabla(\rho_i^+ \phi_i)| \big) \,d\fm \\
&\le \lim_{i \to \infty} \int_M (\rho_i^- +\rho_i^+)|\nabla \phi_i| \,d\fm
 +\liminf_{i \to \infty} \int_M \big( |\nabla \rho_i^-|+|\nabla \rho_i^+| \big) \phi_i \,d\fm \\
&\le \sP(A) +\lim_{i \to \infty} \int_{A \cap \{ \hat{r}<u<\hat{r}+i^{-1} \}} 2i \,d\fm \\
&= \sP(A) +2|A \cap u^{-1}(\hat{r})|.
\end{align*}

Now, it follows from Lemma~\ref{lm:conc} that
$\cI''_{(\R,\bm{\gamma})} \le -\cI_{(\R,\bm{\gamma})}(\theta)^{-1}$ on $(0,\theta]$ (since $\theta<1/2$),
which implies
\[ \sP(A^-_{\hat{r}})
 \ge \cI_{(\R,\bm{\gamma})}\big( \fm(A^-_{\hat{r}}) \big)
 \ge \frac{\fm(A^-_{\hat{r}})}{\theta} \cI_{(\R,\bm{\gamma})}(\theta)
 +\frac{1}{2\cI_{(\R,\bm{\gamma})}(\theta)} \bigg( 1-\frac{\fm(A^-_{\hat{r}})}{\theta} \bigg)
 \frac{\fm(A^-_{\hat{r}})}{\theta} \theta^2. \]
Concerning the second term in the RHS, on one hand, we observe from \eqref{eq:r^-} that
\[ \fm(A^-_{\hat{r}}) \ge (\theta -\sqrt{\delta}) \nu(Q_c \cap Q_{\ell}^-)
 \ge \frac{\theta}{2} \nu(Q_c \cap Q_{\ell}^-). \]
On the other hand, we similarly find
\[ \fm(A^-_{\hat{r}}) =\theta -\fm(A^+_{\hat{r}})
 \le \theta -\frac{\theta}{2} \nu(Q_c \cap Q_{\ell}^+). \]
Therefore, setting $V:=\min\{ \nu(Q_c \cap Q_{\ell}^-),\nu(Q_c \cap Q_{\ell}^+) \} \le 1/2$,
we obtain
\[ \sP(A^-_{\hat{r}})
 \ge \frac{\fm(A^-_{\hat{r}})}{\theta} \cI_{(\R,\bm{\gamma})}(\theta)
 +\frac{1}{2\cI_{(\R,\bm{\gamma})}(\theta)} \bigg( 1-\frac{V}{2} \bigg) \frac{V}{2} \theta^2. \]
We have a similar inequality for $A^+_{\hat{r}}$ in the same way.
Summing up, we obtain
\[ \sP(A^-_{\hat{r}}) +\sP(A^+_{\hat{r}})
 \ge \cI_{(\R,\bm{\gamma})}(\theta)
 +\frac{1}{\cI_{(\R,\bm{\gamma})}(\theta)} \bigg( 1-\frac{V}{2} \bigg) \frac{V}{2} \theta^2
 \ge \cI_{(\R,\bm{\gamma})}(\theta) +c(\theta) V. \]
Combining this with \eqref{eq:r} and $\cI_{(\R,\bm{\gamma})}(\theta) =\sP(A) -\delta$ yields
\[ \frac{3\delta^{(1-\ve)/(9-3\ve)}}{|a_{\theta}|}
 \ge \sP(A^-_{\hat{r}}) +\sP(A^+_{\hat{r}}) -\sP(A)
 \ge c(\theta)V -\delta \]
and hence, by Proposition~\ref{pr:var_X},
\begin{align*}
\min\{ \nu(Q_{\ell}^-),\nu(Q_{\ell}^+) \}
&\le V +\delta^{(1-\ve)/(9-3\ve)} \\
&\le \frac{1}{c(\theta)} \bigg( \frac{3\delta^{(1-\ve)/(9-3\ve)}}{|a_{\theta}|} +\delta \bigg)
 +\delta^{(1-\ve)/(9-3\ve)} \\
&\le C(\theta) \delta^{(1-\ve)/(9-3\ve)}.
\end{align*}
This completes the proof for $\theta<1/2$.

When $\theta>1/2$, the complement $A^c$ of $A$ satisfies $\sP(A^c)=\sP(A)$
and $\fm(A^c)=1-\theta<1/2$.
Note also that $\cI_{(\R,\bm{\gamma})}(\theta) =\cI_{(\R,\bm{\gamma})}(1-\theta)$ and
$r^{-}_{\fm_{q}}(\theta) = r^{+}_{\fm_{q}}(1-\theta)$, $r^{+}_{\fm_{q}}(\theta) = r^{-}_{\fm_{q}}(1-\theta)$. 
Hence we have, since $E \setminus F =E \cap F^c =F^c \setminus E^c$,
\[ A_q \,\triangle\, (-\infty,r_{\fm_q}^-(\theta)] = A_q^c \,\triangle\, (r_{\fm_q}^-(\theta),\infty)
 = A_q^c \,\triangle\, (r_{\fm_q}^+(1-\theta),\infty) \]
and similarly $A_q \,\triangle\, [r_{\fm_q}^+(\theta),\infty) = A_q^c \,\triangle\, (-\infty,r_{\fm_q}^-(1-\theta))$.
Therefore we can obtain the claim for $A$ by applying the above argument to $A^c$.
$\qedd$
\end{proof}

From the proof of Proposition~\ref{pr:Pr6.4},
we find that $C_9(1-\theta)=C_9(\theta)$ and
$\lim_{\theta \to 1/2} C_9(\theta)=\infty$ (since $a_{1/2}=0$).
Hence the case of $\theta=1/2$ is not covered.

We finally prove our main theorem.
We employ the sub-level and super-level sets of the guiding function $u$ instead of balls in \cite{CMM}.

\begin{theorem}[Quantitative isoperimetry]\label{th:main}
Let $(M,g,\fm)$ be a complete weighted Riemannian manifold
such that $\Ric_{\infty} \ge 1$ and $\fm(M)=1$.
Fix $\theta \in (0,1) \setminus \{1/2\}$ and $\ve \in (0,1)$,
take a Borel set $A \subset M$ with $\fm(A)=\theta$,
and assume that $\sP(A) \le \cI_{(\R,\bm{\gamma})}(\theta)+\delta$ holds
for sufficiently small $\delta>0$ $($relative to $\theta$ and $\ve)$.
Then, for the guiding function $u$ associated with $A$ such that $\int_M u \,d\fm=0$, we have
\begin{equation}\label{eq:Qisop}
{\min}\Big\{ \fm\big( A \,\triangle\, \{ u \le a_{\theta} \} \big),
 \fm\big( A \,\triangle\, \{ u \ge a_{1-\theta} \} \big) \Big\}
 \le C(\theta,\ve) \delta^{(1-\ve)/(9-3\ve)}.
\end{equation}
\end{theorem}

\begin{proof}
We set again $\delta=\delta(A)$.
Thanks to Proposition~\ref{pr:Pr6.4}, we first assume
$\nu(Q_{\ell}^+) \le C_9(\theta) \delta^{(1-\ve)/(9-3\ve)}$.
Then we deduce from Lemmas~\ref{lm:Q_l} and \ref{lm:Lem6.2} that
\[ \nu(Q \setminus Q_{\ell}^-)
 =\nu(Q \setminus Q_{\ell}) +\nu\big(Q_{\ell} \setminus (Q_{\ell}^- \cup Q_{\ell}^+) \big) +\nu(Q_{\ell}^+)
 \le 2\sqrt{\delta} +C_9(\theta) \delta^{(1-\ve)/(9-3\ve)}. \]
Therefore we obtain
\begin{align}
&\fm\big( A \,\triangle\, \{ u \le a_{\theta} \} \big) \nonumber\\
&\le \int_{Q_{\ell}^-} \fm_q \big( A_q \,\triangle\, (-\infty,a_{\theta}] \big) \,\nu(dq)
 +\nu(Q \setminus Q_{\ell}^-) \nonumber\\
&\le \int_{Q_{\ell}^-} \fm_q \big( A_q \,\triangle\, (-\infty,r_{\fm_q}^-(\theta)] \big) \,\nu(dq)
 +\int_{Q_{\ell}^-} \fm_q \big( (-\infty,a_{\theta}] \,\triangle\, (-\infty,r_{\fm_q}^-(\theta)] \big) \,\nu(dq)
 \nonumber\\
&\quad +\nu(Q \setminus Q_{\ell}^-) \nonumber\\
&\le \int_{Q_{\ell}^-} \fm_q \big( (-\infty,a_{\theta}] \,\triangle\, (-\infty,r_{\fm_q}^-(\theta)] \big) \,\nu(dq)
 +3\sqrt{\delta} +C_9(\theta) \delta^{(1-\ve)/(9-3\ve)}. \label{eq:main1}
\end{align}
In order to estimate the first term, we recall from Proposition~\ref{pr:var_X} that
$|a_{\theta}-r_{\fm_q}^-(\theta)| \le C_8(\theta,\ve)\delta^{(1-\ve)/(9-3\ve)}$ for $q \in Q_c \cap Q_{\ell}$.
This implies
\begin{align*}
\fm_q \big( (-\infty,a_{\theta}] \,\triangle\, (-\infty,r_{\fm_q}^-(\theta)] \big)
&= \fm_q \Big( \big( \min\{ a_{\theta},r_{\fm_q}^-(\theta) \},\max\{ a_{\theta},r_{\fm_q}^-(\theta) \} \big] \Big) \\
&\le C(\theta,\ve)\delta^{(1-\ve)/(9-3\ve)}
\end{align*}
for $q \in Q_c \cap Q_{\ell}$.
Substituting this into \eqref{eq:main1}, we obtain
\begin{align*}
&\fm\big( A \,\triangle\, \{ u \le a_{\theta} \} \big) \\
&\le C(\theta,\ve)\delta^{(1-\ve)/(9-3\ve)}
 +\nu(Q_{\ell}^- \setminus Q_c) +3\sqrt{\delta} +C_9(\theta) \delta^{(1-\ve)/(9-3\ve)} \\
&\le C(\theta,\ve) \delta^{(1-\ve)/(9-3\ve)}.
\end{align*}

In the case of $\nu(Q_{\ell}^-) \le C_9(\theta) \delta^{(1-\ve)/(9-3\ve)}$,
we similarly have $\fm(A \,\triangle\, \{ u \ge a_{1-\theta} \}) \le C(\theta,\ve) \delta^{(1-\ve)/(9-3\ve)}$.
This completes the proof.
$\qedd$
\end{proof}

We conclude with several remarks and open problems related to Theorem~\ref{th:main}.

\begin{remark}\label{rm:main}
\begin{enumerate}[(a)]
\item\label{main-u}
If we assert only the existence of `some' $1$-Lipschitz function $u$ enjoying \eqref{eq:Qisop},
then one can merely take $u(x):=d(A,x)+a_{\theta}$.
Therefore the novelty of Theorem~\ref{th:main} lies in the construction of $u$
as the guiding function of the needle decomposition.
By construction the guiding function $u$ seems closely related to the Busemann function.
When there is a \emph{straight line} $\eta:\R \lra M$
(meaning that $d(\eta(s),\eta(t))=|s-t|$ for all $s,t \in \R$),
the associated \emph{Busemann function} $\mathbf{b}:M \lra \R$ is defined by
\[ \mathbf{b}(x) :=\lim_{t \to \infty} \big\{ t-d\big( x,\eta(t) \big) \big\}. \]
By construction $\mathbf{b}$ is $1$-Lipschitz
and sometimes regarded as `a distance function from infinity'.
In Cheeger--Gromoll-type splitting theorems
(under $\Ric_N \ge 0$, see also \eqref{main-split} below),
we show that $\mathbf{b}$ is totally geodesic and $M$ is split into $\R \times \Sigma$,
where $\{t\} \times \Sigma =\mathbf{b}^{-1}(t)$ and $\eta_x(t):=(t,x)$
is a straight line for every $x \in \Sigma$.
This is a similar phenomenon to the rigidity of the Bakry--Ledoux isoperimetric inequality
(under $\Ric_{\infty} \ge K>0$) in Theorem~\ref{th:Morgan}, where the guiding function
plays a similar role to the Busemann function (see \cite{Ma2} for details).
Going back to our quantitative investigation,
the guiding function $u$ shares several properties with the Busemann function:
$u$ is $1$-Lipschitz, most needles are long in both directions
($\lim_{\delta \to 0}S=-\infty$ and $\lim_{\delta \to 0}T=\infty$ in Proposition~\ref{pr:psi}),
and the direction of most needles are the same (Proposition~\ref{pr:Pr6.4}).
When, for instance, some needle is a straight line, one may relate the associated Busemann function
with the guiding function and obtain \eqref{eq:Qisop} in terms of that Busemann function.
In this direction, moreover,
one could expect an `almost splitting theorem' as metric measure spaces,
namely $(M,g,\fm)$ is close to the product space $(\R,|\cdot|,\bm{\gamma}) \times Y$ in some sense
(even when there is no infinite needle).
This is an interesting and challenging problem, let us recall that
Gromov's precompactness theorem (\cite[\S 5.A]{Gr}) does not apply under $\Ric_{\infty} \ge K>0$.

\item\label{main-split}
In comparison with the Cheeger--Gromoll-type splitting theorem under $\Ric_{\infty} \ge 0$
in \cite{Li,FLZ}, we remark that the upper boundedness of the weight function $\Psi$
was not assumed in Theorem~\ref{th:main}.
In the splitting theorem we claim that the space splits off the real line
endowed with the Lebesgue measure,
and hence an upper bound of $\Psi$ is necessary to rule out Gaussian spaces
(and hyperbolic spaces with very convex weight functions).
Compare this with the rigidity results under $\Ric_{\infty} \ge K>0$
in Theorems~\ref{th:CZ}, \ref{th:Morgan}.

\item\label{main-F}
Since the needle decomposition is available also for Finsler manifolds by \cite{CM1,Oneedle},
one can prove the analogue of Theorem~\ref{th:main} for reversible Finsler manifolds verbatim.
In the non-reversible case, however, the needle decomposition does not provide
the sharp isoperimetric inequality and it is unclear if one can generalize Theorem~\ref{th:main}.
See \cite{Oneedle} for more details on the non-reversible situation,
and \cite{Oisop} for a derivation of the sharp Bakry--Ledoux isoperimetric inequality
for non-reversible Finsler manifolds.

\item\label{main-RCD}
In Theorem~\ref{th:main} we restrict ourselves to weighted Riemannian manifolds
since the needle decomposition is not yet known for metric measure spaces satisfying
$\CD(1,\infty)$ or $\RCD(1,\infty)$.
We refer to \cite{AM} for the Bakry--Ledoux isoperimetric inequality on $\RCD(1,\infty)$-spaces.

\item\label{main-1/2}
There are two open problems related to Theorem~\ref{th:main}.
The first one is the case of $\theta=1/2$.
The condition $\theta \neq 1/2$ was used only in Proposition~\ref{pr:Pr6.4},
where we showed that one of $Q_{\ell}^-$ and $Q_{\ell}^+$ has a small volume.
If this step is established in some other way,
then all the other steps of the proof work and we can obtain Theorem~\ref{th:main} for $\theta=1/2$.

\item\label{main-sqrt}
Another open problem is the optimal order of $\delta$ in \eqref{eq:Qisop}.
Our estimate $\delta^{(1-\ve)/(9-3\ve)}$ seems not optimal at all and,
compared with the case of Gaussian spaces (recall \eqref{eq:Gauss}),
the optimal order is likely $\sqrt{\delta}$.
We remark that the optimal order is not known also for $\CD(N-1,N)$-spaces studied in \cite{CMM}
($N \in (1,\infty)$), where they obtained $\delta^{N/(N^2+2N-1)}$ depending on $N$
(recall \eqref{eq:CMM}).

\item\label{main-Wass}
Inspired by \cite{DF,CF}, we expect that the push-forward measure $u_* \fm$
is close to $\bm{\gamma}$ in the Wasserstein distance $W_1$ or $W_2$ over $\R$.
We may make use of the Talagrand inequality
$W_2(u_* \fm,\bm{\gamma})^2 \le 2\Ent_{\bm{\gamma}}(u_* \fm)$
(recall Subsection~\ref{ssc:rLSI}).
\end{enumerate}
\end{remark}
\medskip

{\it Acknowledgements}.
We thank Fabio Cavalletti and Max Fathi for discussions during the workshop
``Geometry and Probability'' in Osaka (2019).
SO was supported in part by JSPS Grant-in-Aid for Scientific Research (KAKENHI) 19H01786.

\end{document}